\documentclass{amsart}
\usepackage{amsmath,amssymb,amsthm}
\usepackage{graphicx}
\usepackage{longtable}
\usepackage{stmaryrd}
\usepackage{tikz}
\usetikzlibrary{calc}

\newtheorem{defi}{Definition}[section]
\newtheorem{theo}{Theorem}[section]
\newtheorem{prop}[theo]{Proposition}
\newtheorem{lemm}[theo]{Lemma}
\newtheorem{cor}[theo]{Corollary}
\newtheorem{rem}[theo]{Remark}

\makeatletter
    
    \@addtoreset{equation}{section}
\makeatother
\makeatletter
  \newcommand{\subsubsubsection}{\@startsection{paragraph}{4}{\z@}%
    {1.0\Cvs \@plus.5\Cdp \@minus.2\Cdp}%
    {.1\Cvs \@plus.3\Cdp}%
    {\reset@font\sffamily\normalsize}
  }
  \makeatother
\allowdisplaybreaks[4]

\def\${|\!|\!|}
\newcommand{\rs}{\varodot}
\newcommand{\pl}{\varolessthan}
\newcommand{\pr}{\varogreaterthan}

\newcommand{\conj}[1]{\overline{#1}}

\def\com{\text{{\bf com}}}

\def\B{\mathcal{B}}
\def\L{\mathcal{L}}
\def\E{\mathcal{E}}

\def\R{\mathbb{R}}
\def\T{\mathbb{T}}
\def\Z{\mathbb{Z}}
\def\C{\mathbb{C}}

\newcommand{\Ex}{\mathbb{E}}

\newcommand{\drawC}[2]{
\draw[very thick] #1 -- #2;
\filldraw #1 circle (0.8pt);
\filldraw #2 circle (1.6pt);}

\newcommand{\drawD}[2]{
\draw[thin] #1 circle (1pt);
\draw[thin,double] #1 -- #2;
\filldraw[white] #1 circle (0.2pt);
\draw[draw=black,fill=white, thin] #2 circle (1.6pt);}

\newcommand{\drawI}[2]{
\draw[very thick,black] #1 -- #2;
\filldraw[black] #1 circle (0.8pt);
\filldraw[black] #2 circle (0.8pt);}

\newcommand{\drawJ}[2]{
\filldraw[black] #1 circle (1pt);
\filldraw[black] #1 circle (1pt);
\draw[thin,double] #1 -- #2;
\filldraw[white] #1 circle (0.2pt);
\filldraw[white] #2 circle (0.2pt);}

\newcommand{\putrs}[1]{
\filldraw[white] #1 circle (2.5pt);
\filldraw[black] #1 circle (0.4pt);
\draw[draw=black] #1 circle (2.5pt);
}

\newcommand{\IX}{{\,
\begin{tikzpicture}[baseline=0.5pt, scale=6/12]
\coordinate (O1) at (0,0);
\coordinate (X1) at (0pt,12pt);
\drawC{(O1)}{(X1)}
\end{tikzpicture}\,}}

\newcommand{\JY}{{\,
\begin{tikzpicture}[baseline=0.5pt, scale=6/12]
\coordinate (O1) at (0,0);
\coordinate (X1) at (0pt,12pt);
\drawD{(O1)}{(X1)}
\end{tikzpicture}\,}}

\newcommand{\CC}{{\,
\begin{tikzpicture}[baseline=0.5pt, scale=6/12]
\coordinate (O1) at (0,0);
\coordinate (X1) at (-4pt,12pt);
\coordinate (X2) at (+4pt,12pt);
\drawC{(O1)}{(X1)}
\drawC{(O1)}{(X2)}
\end{tikzpicture}\,}}

\newcommand{\CD}{{\,
\begin{tikzpicture}[baseline=0.5pt, scale=6/12]
\coordinate (O1) at (0,0);
\coordinate (X1) at (-4pt,12pt);
\coordinate (X2) at (+4pt,12pt);
\drawC{(O1)}{(X1)}
\drawD{(O1)}{(X2)}
\end{tikzpicture}\,}}

\newcommand{\CCD}{{\,
\begin{tikzpicture}[baseline=0.5pt, scale=6/12]
\coordinate (O1) at (0,0);
\coordinate (X1) at (-6pt,11pt);
\coordinate (X2) at (+0pt,12pt);
\coordinate (X3) at (+6pt,11pt);
\drawC{(O1)}{(X1)}
\drawC{(O1)}{(X2)}
\drawD{(O1)}{(X3)}
\end{tikzpicture}\,}}

\newcommand{\ICCD}{{\,
\begin{tikzpicture}[baseline=1pt, scale=6/12]
\coordinate (O1) at (0pt,0pt);
\coordinate (O2) at (0pt,7pt);
\coordinate (X1) at (-5pt,14pt);
\coordinate (X2) at (+0pt,15pt);
\coordinate (X3) at (+5pt,14pt);
\drawC{(O2)}{(X1)}
\drawC{(O2)}{(X2)}
\drawD{(O2)}{(X3)}
\drawI{(O1)}{(O2)}
\end{tikzpicture}\,}}

\newcommand{\JDDC}{{\,
\begin{tikzpicture}[baseline=1pt, scale=6/12]
\coordinate (O1) at (0pt,0pt);
\coordinate (O2) at (0pt,7pt);
\coordinate (X1) at (-5pt,14pt);
\coordinate (X2) at (+0pt,15pt);
\coordinate (X3) at (+5pt,14pt);
\drawD{(O2)}{(X1)}
\drawD{(O2)}{(X2)}
\drawC{(O2)}{(X3)}
\drawJ{(O1)}{(O2)}
\end{tikzpicture}\,}}

\newcommand{\ICC}{{\,
\begin{tikzpicture}[baseline=1pt, scale=6/12]
\coordinate (O1) at (0pt,0pt);
\coordinate (O2) at (0pt,7pt);
\coordinate (X1) at (-3pt,15pt);
\coordinate (X3) at (+3pt,15pt);
\drawC{(O2)}{(X1)}
\drawC{(O2)}{(X3)}
\drawI{(O1)}{(O2)}
\end{tikzpicture}\,}}

\newcommand{\JDD}{{\,
\begin{tikzpicture}[baseline=1pt, scale=6/12]
\coordinate (O1) at (0pt,0pt);
\coordinate (O2) at (0pt,7pt);
\coordinate (X1) at (-3pt,15pt);
\coordinate (X3) at (+3pt,15pt);
\drawD{(O2)}{(X1)}
\drawD{(O2)}{(X3)}
\drawJ{(O1)}{(O2)}
\end{tikzpicture}\,}}

\newcommand{\ICD}{{\,
\begin{tikzpicture}[baseline=1pt, scale=6/12]
\coordinate (O1) at (0pt,0pt);
\coordinate (O2) at (0pt,7pt);
\coordinate (X1) at (-3pt,15pt);
\coordinate (X3) at (+3pt,15pt);
\drawC{(O2)}{(X1)}
\drawD{(O2)}{(X3)}
\drawI{(O1)}{(O2)}
\end{tikzpicture}\,}}

\newcommand{\JDC}{{\,
\begin{tikzpicture}[baseline=1pt, scale=6/12]
\coordinate (O1) at (0pt,0pt);
\coordinate (O2) at (0pt,7pt);
\coordinate (X1) at (-3pt,15pt);
\coordinate (X3) at (+3pt,15pt);
\drawD{(O2)}{(X1)}
\drawC{(O2)}{(X3)}
\drawJ{(O1)}{(O2)}
\end{tikzpicture}\,}}

\newcommand{\CDICD}{{\,
\begin{tikzpicture}[baseline=1pt, scale=6/12]
\coordinate (O1) at (0pt,0pt);
\coordinate (O2) at (0pt,7pt);
\coordinate (X1) at (-3pt,15pt);
\coordinate (X3) at (+3pt,15pt);
\coordinate (Y1) at (-5pt,6pt);
\coordinate (Y2) at (+5pt,6pt);
\drawC{(O2)}{(X1)}
\drawD{(O2)}{(X3)}
\drawI{(O1)}{(O2)}
\drawC{(O1)}{(Y1)}
\drawD{(O1)}{(Y2)}
\putrs{(O1)}
\end{tikzpicture}\,}}

\newcommand{\CDICC}{{\,
\begin{tikzpicture}[baseline=1pt, scale=6/12]
\coordinate (O1) at (0pt,0pt);
\coordinate (O2) at (0pt,7pt);
\coordinate (X1) at (-3pt,15pt);
\coordinate (X3) at (+3pt,15pt);
\coordinate (Y1) at (-5pt,6pt);
\coordinate (Y2) at (+5pt,6pt);
\drawC{(O2)}{(X1)}
\drawC{(O2)}{(X3)}
\drawI{(O1)}{(O2)}
\drawC{(O1)}{(Y1)}
\drawD{(O1)}{(Y2)}
\putrs{(O1)}
\end{tikzpicture}\,}}

\newcommand{\CCJDC}{{\,
\begin{tikzpicture}[baseline=1pt, scale=6/12]
\coordinate (O1) at (0pt,0pt);
\coordinate (O2) at (0pt,7pt);
\coordinate (X1) at (-3pt,15pt);
\coordinate (X3) at (+3pt,15pt);
\coordinate (Y1) at (-5pt,6pt);
\coordinate (Y2) at (+5pt,6pt);
\drawD{(O2)}{(X1)}
\drawC{(O2)}{(X3)}
\drawJ{(O1)}{(O2)}
\drawC{(O1)}{(Y1)}
\drawC{(O1)}{(Y2)}
\putrs{(O1)}
\end{tikzpicture}\,}}

\newcommand{\CCJDD}{{\,
\begin{tikzpicture}[baseline=1pt, scale=6/12]
\coordinate (O1) at (0pt,0pt);
\coordinate (O2) at (0pt,7pt);
\coordinate (X1) at (-3pt,15pt);
\coordinate (X3) at (+3pt,15pt);
\coordinate (Y1) at (-5pt,6pt);
\coordinate (Y2) at (+5pt,6pt);
\drawD{(O2)}{(X1)}
\drawD{(O2)}{(X3)}
\drawJ{(O1)}{(O2)}
\drawC{(O1)}{(Y1)}
\drawC{(O1)}{(Y2)}
\putrs{(O1)}
\end{tikzpicture}\,}}

\newcommand{\CDICCD}{{\,
\begin{tikzpicture}[baseline=1pt, scale=6/12]
\coordinate (O1) at (0pt,0pt);
\coordinate (O2) at (0pt,7pt);
\coordinate (X1) at (-5pt,14pt);
\coordinate (X2) at (+0pt,15pt);
\coordinate (X3) at (+5pt,14pt);
\coordinate (Y1) at (-7pt,6pt);
\coordinate (Y2) at (+7pt,6pt);
\drawC{(O2)}{(X1)}
\drawC{(O2)}{(X2)}
\drawD{(O2)}{(X3)}
\drawI{(O1)}{(O2)}
\drawC{(O1)}{(Y1)}
\drawD{(O1)}{(Y2)}
\putrs{(O1)}
\end{tikzpicture}\,}}

\newcommand{\CCJDDC}{{\,
\begin{tikzpicture}[baseline=1pt, scale=6/12]
\coordinate (O1) at (0pt,0pt);
\coordinate (O2) at (0pt,7pt);
\coordinate (X1) at (-5pt,14pt);
\coordinate (X2) at (+0pt,15pt);
\coordinate (X3) at (+5pt,14pt);
\coordinate (Y1) at (-7pt,6pt);
\coordinate (Y2) at (+7pt,6pt);
\drawD{(O2)}{(X1)}
\drawD{(O2)}{(X2)}
\drawC{(O2)}{(X3)}
\drawJ{(O1)}{(O2)}
\drawC{(O1)}{(Y1)}
\drawC{(O1)}{(Y2)}
\putrs{(O1)}
\end{tikzpicture}\,}}

\newcommand{\CICCD}{{\,
\begin{tikzpicture}[baseline=1pt, scale=6/12]
\coordinate (O1) at (0pt,0pt);
\coordinate (O2) at (0pt,7pt);
\coordinate (X1) at (-5pt,14pt);
\coordinate (X2) at (+0pt,15pt);
\coordinate (X3) at (+5pt,14pt);
\coordinate (Y2) at (+7pt,6pt);
\drawC{(O2)}{(X1)}
\drawC{(O2)}{(X2)}
\drawD{(O2)}{(X3)}
\drawI{(O1)}{(O2)}
\drawC{(O1)}{(Y2)}
\putrs{(O1)}
\end{tikzpicture}\,}}

\newcommand{\DICCD}{{\,
\begin{tikzpicture}[baseline=1pt, scale=6/12]
\coordinate (O1) at (0pt,0pt);
\coordinate (O2) at (0pt,7pt);
\coordinate (X1) at (-5pt,14pt);
\coordinate (X2) at (+0pt,15pt);
\coordinate (X3) at (+5pt,14pt);
\coordinate (Y2) at (+7pt,6pt);
\drawC{(O2)}{(X1)}
\drawC{(O2)}{(X2)}
\drawD{(O2)}{(X3)}
\drawI{(O1)}{(O2)}
\drawD{(O1)}{(Y2)}
\putrs{(O1)}
\end{tikzpicture}\,}}

\def\paper{paper}

\begin{document}

\title[Global well-posedness of CGL equation with a space-time white noise]{Global well-posedness of complex Ginzburg-Landau equation with a space-time white noise}
\author{Masato Hoshino}
\date{\today}
\keywords{Complex Ginzburg-Landau equation, Paracontrolled calculus}
\subjclass{60H15, 82C28}
\address{Waseda University, 3-4-1 Okubo, Shinjuku-ku, Tokyo 169-0072, Japan}
\email{m.hoshino3@kurenai.waseda.jp}
\maketitle

\begin{abstract}
We show the global-in-time well-posedness of the complex Ginzburg-Landau (CGL) equation with a space-time white noise on the 3-dimensional torus. Our method is based on \cite{MW16}, where Mourrat and Weber showed the global well-posedness for the dynamical $\Phi_3^4$ model. We prove a priori $L^{2p}$ estimate for the paracontrolled solution as in the deterministic case \cite{DGL}.
\end{abstract}

\section{Introduction}\label{section:introduction}

In this \paper, we consider the following stochastic complex Ginzburg-Landau (CGL) equation on the three-dimensional torus $\T^3=(\R/\Z)^3$:
\begin{equation}\label{introduction:CGL}
\begin{aligned}
&\partial_tu=(i+\mu)\Delta u-\nu|u|^2u+\lambda u+\xi,\quad t>0,\ x\in\T^3,\\
&u(0,\cdot)=u_0,
\end{aligned}
\end{equation}
where $\mu>0$, $\nu\in\{z\in\C\,;\Re z>0\}$, $\lambda\in\C$, and $\xi$ is a complex space-time white noise, which is a centered Gaussian random field with covariance structure
\begin{align*}
\Ex[\xi(t,x)\xi(s,y)]=0,\quad \Ex[\xi(t,x)\overline{\xi(s,y)}]=\delta(t-s)\delta(x-y).
\end{align*}

The CGL equation appears as a generic amplitude equation near the threshold for an instability in fluid mechanics, as well as in the theory of phase transition in superconductivity. Stochastic CGL equation has also been studied in several settings. In \cite{MBS1,MBS2}, CGL equation on a bounded domain in $\R^d$ with a smeared noise in the spatial variable $x$ or a multiplicative noise was studied, where the global well-posedness of the $L^p$ solutions and the existence and uniqueness of an invariant measure were shown, under some additional assumptions. In \cite{HaiCGL}, CGL equation on the one-dimensional torus with a space-time white noise was studied and similar results were shown. In \cite{KS04}, the authors showed the inviscid limit of the CGL equation \eqref{introduction:CGL} with a noise $\sqrt{\mu}\xi$, where $\xi$ is a smeared noise in $x$, to the nonlinear Schr\"odinger equation as $\mu\downarrow0$. The solutions considered in these studies belong to the space of functions. However, when $d\ge2$ and $\xi$ is a space-time white noise, the solution is expected to have the negative regularity $(\frac{2-d}{2})^{-}$, i.e. $\frac{2-d}{2}-\kappa$ for every $\kappa>0$, so that the nonlinear term $-\nu|u|^2u$ of the CGL equation \eqref{introduction:CGL} is ill-defined.

Recent theories of {\em regularity structures} by Hairer \cite{Hai14} or {\em paracontrolled calculus} by Gubinelli, Imkeller and Perkowski \cite{GIP} made it possible to show the general local well-posedness results for several singular stochastic PDEs. In particular, as well as the dynamical $\Phi_d^4$ model, we can apply these theories to the stochastic CGL equation \eqref{introduction:CGL} with a space-time white noise when $d\le3$. For an application for $d=3$, see \cite{HIN}.

The meaning of the local well-posedness for the equation \eqref{introduction:CGL} is as follows. Let $\eta\in\mathcal{S}(\R^3)$ satisfy $\int_{\R^3}\eta(x)dx=1$ and set $\eta^\epsilon(x)=\epsilon^{-3}\eta(\epsilon^{-1}x)$ for $\epsilon>0$. We consider the smeared noise $\xi^\epsilon(t,x)=(\xi(t)*\eta^\epsilon)(x)$ in $x$ and the suitably renormalized equation:
\begin{equation}\label{introduction:renormalized CGL}
\begin{aligned}
&\partial_tu^\epsilon=(i+\mu)\Delta u^\epsilon-\nu|u^\epsilon|^2u^\epsilon+C^\epsilon u^\epsilon+\xi^\epsilon,\quad t>0,\ x\in\T^3,\\
&u^\epsilon(0,\cdot)=u_0,
\end{aligned}
\end{equation}
where $C^\epsilon$ is a constant depending only on $\epsilon,\mu,\nu,\lambda$ and $\eta$, which behaves as $O(\frac1\epsilon)$ as $\epsilon\downarrow0$. For the precise definition of $C^\epsilon$, see \cite[Sections~3.4 and 5.4]{HIN}. Since $u^\epsilon$ is a continuous function in $(t,x)$, we can define the nonlinear term $-\nu|u^\epsilon|^2u^\epsilon$ in usual sense. In \cite{HIN}, by using the theory of regularity structures and paracontrolled calculus, the authors showed that the sequence $\{u^\epsilon\}_{\epsilon>0}$ converges as $\epsilon\downarrow0$ in the space $\mathcal{C}^{-\frac23+\kappa}$ for every small $\kappa>0$, where $\mathcal{C}^\alpha=\mathcal{B}_{\infty,\infty}^\alpha$ is the complex-valued Besov space on $\T^3$. However, they showed only the convergence up to some random time $T\in(0,\infty]$ and did not study whether $T=\infty$ or not.

The aim of this \paper\ is to show the {\em global-in-time} well-posedness for the equation \eqref{introduction:CGL} using the paracontrolled calculus. We use similar arguments to \cite{MW16}, where Mourrat and Weber showed the global well-posedness for the dynamical $\Phi_3^4$ model:
$$
\partial_tX=\Delta X-X^3+mX+\xi,\quad t>0,\ x\in\T^3,
$$
which is regarded as a real-valued version of the equation \eqref{introduction:CGL}. However, in our setting we need to improve their method as we will explain later. The main result of this paper is formulated as follows.

\begin{theo}\label{introduction:main theorem}
Let $\mu>\frac{1}{2\sqrt{2}}$. Choose sufficiently small $\kappa>0$ depending on $\mu$. For every initial value $u_0\in\mathcal{C}^{-\frac23+\kappa}$, the sequence $\{u^\epsilon\}_{\epsilon>0}$ of the solution of \eqref{introduction:renormalized CGL} has a limit $u\in C([0,\infty),\mathcal{C}^{-\frac23+\kappa})$, that is, for every $T>0$ we have
\begin{align*}
\lim_{\epsilon\downarrow0}\|u^\epsilon-u\|_{C([0,T],\mathcal{C}^{-\frac23+\kappa})}=0
\end{align*}
in probability. The limit $u$ is independent of the choice of the mollifier $\eta$.
\end{theo}

We reformulate the above theorem more precisely in Theorem \ref{localexistence+:main theorem} below.
 
We briefly explain the outline of the proof of Theorem \ref{introduction:main theorem}. If the noise $\xi$ is a continuous function in $(t,x)$, then the solution $u$ of the equation \eqref{introduction:CGL} would satisfy a priori $L^{2p}$ inequality:
$$
\sup_{0\le t\le T}\|u(t)\|_{L^{2p}}\le\|u_0\|_{L^{2p}}+C
$$
when the condition
\begin{align}\label{introduction:p<f(mu)}
1<p<1+\mu(\mu+\sqrt{1+\mu^2})
\end{align}
holds. See Proposition \ref{w:L^2p control of w from PDE} below or \cite[Section~4]{DGL}. However, since $u$ is distribution-valued in the present case, the $L^{2p}$ norm of $u$ diverges. In order to overcome this difficulty, we use a similar method to \cite{MW16}. Our method consists of the following three steps.

(1) Following the general theory of the paracontrolled calculus, we divide the solution $u$ into the sum
$$
u=\IX-\nu\ICCD+v+w,
$$
where $\IX$ and $\ICCD$ are stochastic processes explicitly defined, $v$ is the solution of the linear equation which contains $w$ as a coefficient, and $w$ is the regular term and solves the nonlinear equation of the form
$$
\partial_tw=(i+\mu)w-\nu|w|^2w+\cdots.
$$
For the precise definition of $(v,w)$, see the system \eqref{solution:CGLsystem} below.

(2) From the definition, a suitable norm of $v$ is controlled by a suitable norm of $w$. Hence it is sufficient to control only $w$ in some suitable norms. Since $w$ is sufficiently regular, we can apply the method of $L^{2p}$ inequality explained above to $w$ when the condition \eqref{introduction:p<f(mu)} holds. However, from the definition of the system \eqref{solution:CGLsystem}, we also need the control of $w$ in the $\B_{\frac{2p+2}3,\infty}^{1+2\kappa}$ norm. The second goal is to show a priori $L^1[0,T]$ estimate
\begin{align}\label{introduction:L^1[0,T]}
\int_0^T\Bigl(\|v(t)\|_{\B_{2p+2,\infty}^{\frac12+\kappa}}^{2p+2}+\|w(t)\|_{L^{2p+2}}^{2p+2}+\|w(t)\|_{\B_{\frac{2p+2}3,\infty}^{1+2\kappa}}^{\frac{2p+2}3}\Bigr)dt\le C'
\end{align}
for every $p>1$ and every small $\kappa>0$, see Theorem \ref{int:apriori estimate on L^1} below. Note that the similar estimate to above was obtained in \cite[Theorem~6.1]{MW16} for $p=2$.

(3) The final step is to improve the above $L^1[0,T]$ estimate into a priori $L^\infty[0,T]$ estimate
\begin{align}\label{introduction:L^infty[0,T]}
\sup_{0\le t\le T}\Bigl(\|v(t)\|_{\B_{2p+2,\infty}^{\frac12+\kappa}}+\|w(t)\|_{\B_{\frac{2p+2}3,\infty}^{\frac32-2\kappa}}\Bigr)\le C''
\end{align}
for every $p$ as close to $1$ as possible. We will see that the above estimate holds for every $p>\frac32$ in Theorems \ref{global:L^infty of v} and \ref{global:L^infty of w} below. As a result, we will obtain the global well-posedness for the equation \eqref{introduction:CGL} for every $\mu>\frac1{2\sqrt{2}}$, because the condition \eqref{introduction:p<f(mu)} is assumed.

Now we point out two differences in the proof of Theorem \ref{introduction:main theorem} from the arguments of \cite{MW16}. One difference is in the step (2). Since the condition \eqref{introduction:p<f(mu)} requires $\mu$ to be large depending on the value of $p$, we need to prove the $L^1[0,T]$ estimate \eqref{introduction:L^1[0,T]} for $p$ as close to $1$ as possible. Although Mourrat and Weber \cite{MW16} showed the $L^1[0,T]$ estimate \eqref{introduction:L^1[0,T]} for $p=2$ for the dynamical $\Phi_3^4$ model, it is not straightforward to rewrite their method for general $p>1$. Especially, the inequality (4.1) in \cite[Theorem~4.1]{MW16} was rather complicated, so that the estimate \eqref{introduction:L^1[0,T]} was shown only for $p=2$. In this paper, we have reviewed their result and rewrite it into a simpler form (Theorem \ref{dw:goal}), where the last two terms of the equality (4.1) in \cite[Theorem~4.1]{MW16} disappear. As a result, we can show the estimate \eqref{introduction:L^1[0,T]} for every $p>1$.

The other one is in the step (3). We can improve the $L^1[0,T]$ estimate \eqref{introduction:L^1[0,T]} into the $L^\infty[0,T]$ estimate \eqref{introduction:L^infty[0,T]} by using the Young's inequality repeatedly, see Section \ref{section:global} for details. Although this iteration was done four times in \cite[Table~2]{MW16}, we will see that we need more iterations as $p$ gets closer to $\frac32$ in our setting. Indeed, the number of the iterations diverges as $p\downarrow\frac32$. This argument works due to the two estimates given in Lemma \ref{global:lemm:iterating Young} below, which mean to what extent the cubic nonlinearity of the equation \eqref{introduction:CGL} can be {\em weakened}. In the present case, the exponent of the nonlinearity is weakened from ``$3$" to ``$\frac{12}7$". We believe that the condition $p>\frac32$ is optimal as long as we use this method.

This \paper\ is organized as follows. In Section \ref{section:para}, we recall some basic notions and results of the paracontrolled calculus. In Section \ref{section:solution}, we reformulate the equation \eqref{introduction:CGL} as a system of equations of $(v,w)$ and give the local well-posedness result. In the rest of this \paper, we prove the global well-posedness by the method explained above. In Section \ref{section:apriori_v}, we control a norm of $v$ by a norm of $w$. In Section \ref{section:apriori_w}, we apply the method of the $L^{2p}$ inequality to $w$ for every $p>1$, which is completed in Section \ref{section:apriori_dw}. In Section \ref{section:integrable}, we prove a priori $L^1[0,T]$ estimate \eqref{introduction:L^1[0,T]}. In Section \ref{section:global}, we finally obtain a priori $L^\infty[0,T]$ estimate \eqref{introduction:L^infty[0,T]}.

\section{Paracontrolled calculus}\label{section:para}

We recall some basic notions and results from \cite{GIP, MW16}. In what follows, for two functions $A=A(\lambda)$ and $B=B(\lambda)$ of a variable $\lambda$, we write $A\lesssim B$ if there exists a constant $c>0$ independent of $\lambda$ and one has $A\le cB$. We write $A\lesssim_\mu B$ if we want to emphasize that the constant $c$ depends on another parameter $\mu$.

\subsection{Notations}

First we recall the definition of the Besov spaces on $\T^3$ from \cite[Section~2]{BCD}. For $f,g\in L^2=L^2(\T^3,\C)$, we define the bilinear functional
$$
\langle f,g\rangle=\int_{\T^3}f(x)g(x)dx.
$$
Note that we do \emph{not} take the complex conjugate. We write $\mathbf{e}_k(x)=e^{2\pi ik\cdot x}\in L^2$ for $k\in\Z^3$ and denote by $\hat{u}(k)=\langle u,\mathbf{e}_{-k}\rangle$ the Fourier transform of $u\in L^2$. The {\em Besov space} $\B_{p,q}^\alpha$ is defined via Littlewood-Paley theory. Let $\{\rho_j\}_{j=-1}^\infty\subset C_0^\infty(\R^3)$ be a dyadic partition of unity, i.e.
\begin{enumerate}
\setlength{\itemsep}{1mm}
\item $\rho_{-1}$ and $\rho_0$ are radial smooth functions taking values in $[0,1]$.
\item $\text{supp}(\rho_{-1})\subset B(0,\frac{4}{3})$ and $\text{supp}(\rho_0)\in B(0,\frac{8}{3})\setminus B(0,\frac{3}{4})$, where $B(x,r)$ is the open ball in $\R^3$ of center $x$ and radius $r$.
\item $\rho_j=\rho_0(2^{-j}\cdot)$ for every $j\ge0$.
\item $\sum_{j=-1}^\infty\rho_j\equiv1$.
\end{enumerate}
Let $\triangle_j$ be the operator on $L^2$ defined by $\triangle_ju=\sum_{k\in\Z^3}\rho_j(k)\hat{u}(k)\mathbf{e}_k$. For every $\alpha\in\R$ and $p,q\in[1,\infty]$, we define the $\mathcal{B}_{p,q}^\alpha$ norm of $u\in L^2$ by
$$
\|u\|_{\B_{p,q}^\alpha}=\left\|\left(2^{j\alpha}\|\triangle_ju\|_{L^p(\T^3)}\right)\right\|_{l^q(\{-1\}\cup\mathbb{N})}.
$$
We define the space $\B_{p,q}^\alpha$ as the completion of $C^\infty(\T^3,\C)$ under the $\B_{p,q}^\alpha$ norm. This definition ensures that $\B_{p,q}^\alpha$ is separable and that the heat semigroup $(e^{t(i+\mu)\Delta})_{t\ge0}$ is strongly continuous on $\B_{p,q}^\alpha$ even if $q=\infty$, see \cite[Remark~3.13]{MW15}. We use the brief notation $\B_p^\alpha=\B_{p,\infty}^\alpha$ when $q=\infty$.

We formally define the Bony's {\em paraproduct}
$$
u\pl v=v\pr u=\sum_{i\le j-2}\triangle_iu\triangle_jv
$$
and the {\em resonant}
$$
u\rs v=\sum_{|i-j|\le1}\triangle_iu\triangle_jv.
$$
Note that we have $\overline{u\pl v}=\overline{u}\pl\overline{v}$ and $\overline{u\rs v}=\overline{u}\rs\overline{v}$ since $\overline{\triangle_ju}=\triangle_j\overline{u}$. These operators are well-defined under the assumptions of Proposition \ref{para:paraproduct and resonant} below.

We define several classes of functions from the time interval to the Besov space. Let $\alpha\in\R$ and $\delta\in(0,1]$.
\begin{itemize}
\setlength{\itemsep}{1mm}
\item $C_T\B_\infty^\alpha=C([0,T],\B_\infty^\alpha)$, equipped with the supremum norm
\[\|u\|_{C_T\B_\infty^\alpha}=\sup_{0\le t\le T}\|u(t)\|_{\B_\infty^\alpha}.\]
\item $C_T^\delta\B_\infty^\alpha=C^\delta([0,T],\B_\infty^\alpha)$, equipped with the seminorm
\[\|u\|_{C_T^\delta\B_\infty^\alpha}=\sup_{0\le s<t\le T}\frac{\|u(t)-u(s)\|_{\B_\infty^\alpha}}{|t-s|^\delta}.\]
\item $\L_T^{\alpha,\delta}=C_T\B_\infty^\alpha\cap C_T^\delta\B_\infty^{\alpha-2\delta}$ with the norm $\|\cdot\|_{\L_T^{\alpha,\delta}}=\|\cdot\|_{C_T\B_\infty^\alpha}+\|\cdot\|_{C_T^\delta\B_\infty^{\alpha-2\delta}}$.
\end{itemize}
It is useful to consider the norms which allow singularities at $t=0$. Let $\eta>0$.
\begin{itemize}
\setlength{\itemsep}{1mm}
\item $\E_T^\eta\B_\infty^\alpha=\{u\in C((0,T],\B_\infty^\alpha)\,;\|u\|_{\E_T^\eta\B_\infty^\alpha}<\infty\}$, where
\[\|u\|_{\E_T^\eta\B_\infty^\alpha}=\sup_{0<t\le T}t^\eta\|u(t)\|_{\B_\infty^\alpha}.\]
\item $\E_T^{\eta,\delta}\B_\infty^\alpha=\{u\in C((0,T],\B_\infty^\alpha)\,;\|u\|_{\E_T^{\eta,\delta}\B_\infty^\alpha}<\infty\}$, where
\[\|u\|_{\E_T^{\eta,\delta}\B_\infty^\alpha}=\sup_{0<s<t\le T}s^\eta\frac{\|u(t)-u(s)\|_{\B_\infty^\alpha}}{|t-s|^\delta}.\]
\item $\L_T^{\eta,\alpha,\delta}=\E_T^\eta\B_\infty^\alpha\cap C_T\B_\infty^{\alpha-2\eta}\cap\E_T^{\eta,\delta}\B_\infty^{\alpha-2\delta}$ with the norm $\|\cdot\|_{\L_T^{\eta,\alpha,\delta}}=\|\cdot\|_{\E_T^\eta\B_\infty^\alpha}+\|\cdot\|_{C_T\B_\infty^{\alpha-2\eta}}+\|\cdot\|_{\E_T^{\eta,\delta}\B_\infty^{\alpha-2\delta}}$.
\end{itemize}
When we consider the functions on $[0,\infty)$, we denote by $C\B_\infty^\alpha$ the Fr\'{e}chet space defined by the norms $\{\|\cdot\|_{C_T\B_\infty^\alpha}\}_{T>0}$. We define the spaces $C^\delta\B_\infty^\alpha$ and $\L^{\alpha,\delta}$ similarly.

\subsection{Basic estimates}

We give some basic results without proofs. They are used repeatedly in this \paper.

\begin{prop}\label{para:estimates of Besov norm}
Let $\alpha,\beta\in\R$ and $p,p_1,p_2,q,q_1,q_2\in[1,\infty]$.
\begin{enumerate}
\setlength{\itemsep}{1mm}
\item If $\alpha<\beta$, then $\|u\|_{\B_{p,q}^\alpha}\le\|u\|_{\B_{p,q}^\beta}$. Furthermore, $\|u\|_{\B_{p,1}^\alpha}\lesssim\|u\|_{\B_{p,q}^\beta}$ (\cite[Remark~3.4]{MW15}).
\item If $p_1\le p_2$, then $\|u\|_{\B_{p_1,q}^\alpha}\le\|u\|_{\B_{p_2,q}^\alpha}$.
\item If $q_1\ge q_2$, then $\|u\|_{\B_{p,q_1}^\alpha}\le\|u\|_{\B_{p,q_2}^\alpha}$.
\item $\|u\|_{\B_{p,\infty}^0}\lesssim\|u\|_{L^p}\le \|u\|_{\B_{p,1}^0}$ (\cite[Remark~3.5]{MW15}).
\end{enumerate}
\end{prop}

\begin{prop}[{\cite[Theorem~2.80]{BCD}}]\label{para:interpolation}
For every $\alpha_0,\alpha_1\in\R$, $p_0,p_1,q_0,q_1\in[1,\infty]$ and $\nu\in[0,1]$, we have
$$
\|u\|_{\B_{p,q}^\alpha}\le\|u\|_{\B_{p_0,q_0}^{\alpha_0}}^{1-\nu}\|u\|_{\B_{p_1,q_1}^{\alpha_1}}^\nu,
$$
where $\alpha=(1-\nu)\alpha_0+\nu\alpha_1$, $\frac{1}{p}=\frac{1-\nu}{p_0}+\frac{\nu}{p_1}$, and $\frac{1}{q}=\frac{1-\nu}{q_0}+\frac{\nu}{q_1}$.
\end{prop}

\begin{prop}[{\cite[Proposition~2.76]{BCD}} and {\cite[Proposition~3.23]{MW15}}]
For every $\alpha\in\R$ and $p,p',q,q'\in[1,\infty]$ such that $1=\frac1p+\frac1{p'}=\frac1q+\frac1{q'}$, we have
$$
|\langle f,g\rangle|\lesssim\|f\|_{\B_{p,q}^\alpha}\|g\|_{\B_{p',q'}^{-\alpha}}.
$$
\end{prop}

\begin{prop}[{\cite[Theorem~2.71]{BCD}}]
 For every $\alpha\in\R$, $1\le p_1\le p_2\le\infty$ and $q\in[1,\infty]$, we have
$$
\|u\|_{\B_{p_2,q}^\alpha}\lesssim\|u\|_{\B_{p_1,q}^{\alpha+3\left(\frac{1}{p_1}-\frac{1}{p_2}\right)}}.
$$
\end{prop}

\begin{prop}[{\cite[Proposition~A.6]{MW16}}]\label{para:sobolev into besov}
For every $\alpha\in(0,1]$ and $p\in[1,\infty]$, we have
$$
\|u\|_{\B_p^\alpha}\lesssim\|u\|_{L^p}^{1-\alpha}\|\nabla u\|_{L^p}^\alpha+\|u\|_{L^p},
$$
where $\nabla u=(\partial_1 u,\partial_2 u,\partial_3 u)$ is the gradient of $u$ in the sense of distributions.
\end{prop}

We summarize some important estimates of the paraproduct and the resonant.

\begin{prop}[{\cite[Theorem~3.17]{MW15}}]\label{para:paraproduct and resonant}
Let $p,p_1,p_2,q,q_1,q_2\in[1,\infty]$ be such that $\frac1p=\frac1{p_1}+\frac1{p_2}$ and $\frac1q=\frac1{q_1}+\frac1{q_2}$.
\begin{enumerate}
\setlength{\itemsep}{1mm}
\item For every $\alpha\in\R$, $\|u\pl v\|_{\B_{p,q}^\alpha}\lesssim\|u\|_{L^{p_1}}\|u\|_{\B_{p_2,q}^\alpha}$.
\item For every $\alpha<0$ and $\beta\in\R$, $\|u\pl v\|_{\B_{p,q}^{\alpha+\beta}}\lesssim\|u\|_{\B_{p_1,q_1}^\alpha}\|v\|_{\B_{p_2,q_2}^\beta}$.
\item If $\alpha+\beta>0$, then $\|u\rs v\|_{\B_{p,q}^{\alpha+\beta}}\lesssim\|u\|_{\B_{p_1,q_1}^\alpha}\|v\|_{\B_{p_2,q_2}^\beta}$.
\end{enumerate}
\end{prop}

\begin{prop}[{\cite[Proposition~A.9]{MW16}}]
Let $\alpha<1$, $\beta,\gamma\in\R$ and $p,p_1,p_2,p_3\in[1,\infty]$ be such that $\beta+\gamma<0$, $\alpha+\beta+\gamma>0$ and $\frac1p=\frac1{p_1}+\frac1{p_2}+\frac1{p_3}$. Let $R$ be the trilinear map
$$
R(u,v,w)=(u\pl v)\rs w-u(v\rs w)
$$
defined for $u,v,w\in C^\infty(\T^3,\C)$. Then $R$ is uniquely extended to a continuous trilinear map from $\B_{p_1}^\alpha\times\B_{p_2}^\beta\times\B_{p_3}^\gamma$ to $\B_p^{\alpha+\beta+\gamma}$.
\end{prop}

We summarize the regularizing effects of the heat semigroup $(e^{t(i+\mu)\Delta})_{t\ge0}$ generated by the operator $(i+\mu)\Delta$.

\begin{prop}[{\cite[Propositions~3.11 and 3.12]{MW15}}]
Let $\alpha\in\R$, $p,q\in[1,\infty]$ and $\mu>0$.
\begin{enumerate}
\setlength{\itemsep}{1mm}
\item For every $\delta\ge0$, $\|e^{t(i+\mu)\Delta}u\|_{\B_{p,q}^{\alpha+2\delta}}\lesssim t^{-\delta}\|u\|_{\B_{p,q}^\alpha}$ uniformly over $t>0$.
\item For every $\delta\in[0,1]$, $\|(e^{t(i+\mu)\Delta}-1)u\|_{\B_{p,q}^{\alpha-2\delta}}\lesssim t^\delta\|u\|_{\B_{p,q}^\alpha}$ uniformly over $t>0$.
\end{enumerate}
\end{prop}

\begin{prop}[{\cite[Proposition~A.15]{MW16}}]
Let $\alpha<1$, $\beta\in\R$, $\delta\ge0$, and $p,p_1,p_2\in[1,\infty]$ be such that $\frac1p=\frac1{p_1}+\frac1{p_2}$. Define
\[[e^{t(i+\mu)\Delta},u\pl]v=e^{t(i+\mu)\Delta}(u\pl v)-u\pl e^{t(i+\mu)\Delta}v.\]
Then we have
\[\|[e^{t(i+\mu)\Delta},u\pl]v\|_{\B_p^{\alpha+\beta+2\delta}}\lesssim t^{-\delta}\|u\|_{\B_{p_1}^\alpha}\|v\|_{\B_{p_2}^\beta}\]
uniformly over $t>0$.
\end{prop}

\section{Paracontrolled CGL equation}\label{section:solution}

We reformulate the stochastic CGL equation \eqref{introduction:CGL} based on the paracontrolled calculus approach and give the local well-posedness result. For details, see \cite[Section~4]{HIN}.

\subsection{Definition of the solution}\label{subsection:Definition of the solution}

We explain how to give a meaning to the equation \eqref{introduction:CGL} based on the method in \cite{MW16}. If the regularity is written as $\alpha^-$ or $\alpha^+$, then it can be replaced by $\alpha-\delta$ or $\alpha+\delta$ for every small $\delta>0$.

Let $\L_\mu=\partial_t-\{(i+\mu)\Delta-1\}$ and rewrite \eqref{introduction:CGL} as
$$
\L_\mu u=-\nu u^2\overline{u}+(\lambda+1)u+\xi.
$$
We think of the noise as the leading term and the nonlinear term as its perturbation. Let $\IX$ be the stationary solution of
$$
\L_\mu \IX=\xi,
$$
then $\IX$ has regularity $(-\frac12)^-$. Let $\JY=\overline{\IX}$. Since we cannot define the products
$$
\CC=(\IX)^2,\quad\CD=\IX\JY,\quad\CCD=(\IX)^2\JY
$$
in usual sense, we now assume that the elements $\CC,\CD$ with regularity $(-1)^-$ and $\CCD$ with regularity $(-\frac32)^-$ are given a priori. If we set $u=u_1+\IX$, then we have the equation
\begin{align*}
\L_\mu u_1&=-\nu(u_1+\IX)^2(\overline{u_1}+\JY)+(\lambda+1)(u_1+\IX)\\
&=-\nu(u_1^2\overline{u_1}+u_1^2\JY+2u_1\overline{u_1}\IX+2u_1\CD+\overline{u_1}\CC+\CCD)+(\lambda+1)(u_1+\IX),\\
&=-\nu(2u_1\CD+\overline{u_1}\CC+\CCD)+P(u_1),
\end{align*}
where
$$
P(u_1)=-\nu(u_1^2\overline{u_1}+u_1^2\JY+2u_1\overline{u_1}\IX)+(\lambda+1)(u_1+\IX).
$$
We continue the decomposition. Let $\ICCD$ be the stationary solution of
$$
\L_\mu\ICCD=\CCD,
$$
then $\ICCD$ has regularity $\frac12^-$. Let $\JDDC=\overline{\ICCD}$. If we set $u_1=u_2-\nu\ICCD$, then we have
$$
\L_\mu u_2=-\nu\{2(u_2-\nu\ICCD)\CD+(\overline{u_2}-\overline{\nu}\JDDC)\CC\}+P(u_2-\nu\ICCD).
$$
Here we can write $P(u_2-\nu\ICCD)$ as
$$
P(u_2-\nu\ICCD)=P_0+P_1(u_2)+P_2(u_2)-\nu u_2^2\conj{u_2},
$$
where
\begin{align*}
P_0&=-\nu(-\nu^2\conj{\nu}(\ICCD)^2\JDDC+\nu^2(\ICCD)^2\JY+2\nu\conj{\nu}\ICCD\JDDC\IX)+(\lambda+1)(-\nu\ICCD+\IX),\\
P_1(u_2)&=-\nu\{u_2(2\nu\conj{\nu}\ICCD\JDDC-2\nu\ICCD\JY-2\conj{\nu}\JDDC\IX)+\conj{u_2}(\nu^2(\ICCD)^2-2\nu\ICCD\IX)\}+(\lambda+1)u_2,\\
P_2(u_2)&=-\nu\{u_2^2(-\conj{\nu}\JDDC+\JY)+2u_2\conj{u_2}(-\nu\ICCD+\IX)\}.
\end{align*}
Although we have the ill-defined terms $\ICCD\JY$, $\JDDC\IX$, $\ICCD\IX$, $(\ICCD)^2\JY$ and $\ICCD\JDDC\IX$, they are well-defined if we assume that the elements
$$
\CICCD=\ICCD\rs\IX,\quad\DICCD=\ICCD\rs\JY
$$
with regularity $0^-$ are given a priori. For example, $\ICCD\IX$ is defined by
$$
\ICCD\IX=\ICCD(\pl+\pr)\IX+\CICCD,
$$
and so are $\ICCD\JY$ and $\JDDC\IX$. For $(\ICCD)^2\JY$, since it is formally decomposed as
\begin{align*}
(\ICCD)^2\JY&=\ICCD(\ICCD\JY)=\ICCD(\ICCD\pl\JY)+\ICCD(\DICCD+\ICCD\pr\JY)\\
&=(\ICCD\pl\JY)\rs\ICCD+(\ICCD\pl\JY)(\pl+\pr)\ICCD+\ICCD(\DICCD+\ICCD\pr\JY)\\
&=2\ICCD\DICCD+R(\ICCD,\JY,\ICCD)+(\ICCD\pl\JY)(\pl+\pr)\ICCD+\ICCD(\ICCD\pr\JY),
\end{align*}
we can regard the last expression as a definition of $(\ICCD)^2\JY$. We define $\ICCD\JDDC\IX$ by a similar way.

For the terms $(u_2-\nu\ICCD)\CD$ and $(\overline{u_2}-\overline{\nu}\JDDC)\CC$, however, since $u_2$ is expected to have regularity $1^-$, they are still ill-defined. In order to overcome this problem, we introduce the decomposition $u_2=v+w$, which solve
\begin{align}
\label{solution:v}\L_\mu v&=-\nu\{2(v+w-\nu\ICCD)\pl\CD+(\overline{v}+\overline{w}-\overline{\nu}\JDDC)\pl\CC\}-cv\\
\label{solution:w}\L_\mu w&=-\nu\{2(v+w-\nu\ICCD)(\rs+\pr)\CD+(\conj{v}+\conj{w}-\conj{\nu}\JDDC)(\rs+\pr)\CC\}\\\notag
&\quad+P(v+w-\nu\ICCD)+cv,
\end{align}
where $c>0$ is a sufficiently large constant defined below. Since $w$ is expected to have regularity $\frac32^-$, the resonant terms $w\rs\CD$ and $\conj{w}\rs\CC$ are well-defined. Although the resonant terms
$$
\CDICCD=\ICCD\rs\CD,\quad\CCJDDC=\JDDC\rs\CC
$$
cannot be defined in usual sense, we assume that they are given a priori as elements with regularity $(-\frac12)^-$. In order to define the resonant terms $v\rs\CD$ and $\conj{v}\rs\CC$, we define $\ICD$ and $\ICC$ as the stationary solutions of
$$
\L_\mu\ICD=\CD,\quad\L_\mu\ICC=\CC,
$$
respectively. Then $\ICD$ and $\ICC$ have regularity $1^-$. Let $\JDC=\conj{\ICD}$ and $\JDD=\conj{\ICC}$. The resonant terms $v\rs\CD$ and $\conj{v}\rs\CC$ are well-defined if the resonants
$$
\CDICD=\ICD\rs\CD,\quad
\CDICC=\ICC\rs\CD,\quad
\CCJDC=\JDC\rs\CC,\quad
\CCJDD=\JDD\rs\CC
$$
are given a priori as elements with regularity $0^-$. Indeed, since we can show that the solution $v$ of \eqref{solution:v} has the form
$$
v=-\nu\{2(v+w-\nu\ICCD)\pl\ICD+(\conj{v}+\conj{w}-\conj{\nu}\JDDC)\pl\ICC\}+\com(v,w),
$$
where $\com(v,w)$ has regularity $1^+$ (see Lemma \ref{local:estimate of com}), we can write the resonants $v\rs\CD$ and $\overline{v}\rs\CC$ as
\begin{multline*}
v\rs\CD=-\nu\{2(v+w-\nu\ICCD)\CDICD+(\conj{v}+\conj{w}-\conj{\nu}\JDDC)\CDICC\\
{}+2R(v+w-\nu\ICCD,\ICD,\CD)+R(\conj{v}+\conj{w}-\conj{\nu}\JDDC,\ICC,\CD)\}+\com(v,w)\rs\CD,
\end{multline*}
and
\begin{multline*}
\conj{v}\rs\CC=-\conj{\nu}\{2(\conj{v}+\conj{w}-\conj{\nu}\JDDC)\CCJDC+(v+w-\nu\ICCD)\CCJDD\\
{}+2R(\conj{v}+\conj{w}-\conj{\nu}\JDDC,\JDC,\CC)+R(v+w-\nu\ICCD,\JDD,\CC)\}+\overline{\com(v,w)}\rs\CC.
\end{multline*}
We have completed the definitions of all terms appeared in the system \eqref{solution:v}-\eqref{solution:w}.

Now we summarize the above argument. We have the well-defined system
\begin{equation}\label{solution:CGLsystem}
\begin{aligned}
\L_\mu v&=F(v,w)-cv,\\
\L_\mu w&=G(v,w)+cv
\end{aligned}
\end{equation}
with initial values $(v(0,\cdot),w(0,\cdot))=(v_0,w_0)$, where
\begin{align*}
F(v,w)=-\nu\{2(v+w-\nu\ICCD)\pl\CD+(\overline{v}+\overline{w}-\overline{\nu}\JDDC)\pl\CC\}
\end{align*}
and
\begin{align*}
G(v,w)&=\sum_{i=1}^8G_{(i)}(v,w),\\
G_{(1)}(v,w)&=-\nu(v+w)^2(\conj{v}+\conj{w}),\\
G_{(2)}(v,w)&=P_2(v+w),\\
G_{(3)}(v,w)&=P_1(v+w)-\nu\{(v+w)(-4\nu\CDICD-\conj{\nu}\CCJDD)+(\conj{v}+\conj{w})(-2\conj{\nu}\CCJDC-2\nu\CDICC)\},\\
G_{(4)}&=P_0-\nu\{\nu\ICCD(4\nu\CDICD+\conj{\nu}\CCJDD)+2\conj{\nu}\JDDC(\conj{\nu}\CCJDC+\nu\CDICC)
-2\nu\ICCD(\rs+\pr)\CD\\
&\quad-\conj{\nu}\JDDC(\rs+\pr)\CC
+4\nu^2R(\ICCD,\ICD,\CD)+2\nu\conj{\nu}R(\JDDC,\ICC,\CD)\\
&\quad+2\conj{\nu}^2R(\JDDC,\JDC,\CC)+\nu\conj{\nu}R(\ICCD,\JDD,\CC)\},\\
G_{(5)}(v,w)&=\nu^2\{4R(v+w,\ICD,\CD)+2R(\conj{v}+\conj{w},\ICC,\CD)\}\\
&\quad+\nu\conj{\nu}\{2R(\conj{v}+\conj{w},\JDC,\CC)+R(v+w,\JDD,\CC)\},\\
G_{(6)}(v,w)&=-\nu\{2\com(v,w)\rs\CD+\conj{\com(v,w)}\rs\CC\},\\
G_{(7)}(v,w)&=-\nu(2w\rs\CD+\conj{w}\rs\CC),\\
G_{(8)}(v,w)&=-\nu\{2(v+w)\pr\CD+(\conj{v}+\conj{w})\pr\CC\}.
\end{align*}

We define the set of drivers which should be given a priori.

\begin{defi}
Let $\kappa>0$. We call a vector of distribution-valued functions on $[0,\infty)$ of the form
\begin{align}
\mathbb{X}&=(\IX,\ICCD,\CICCD,\DICCD,\CD,\CC,\ICD,\ICC,\CDICD,\CDICC,\CCJDC,\CCJDD,\CDICCD,\CCJDDC)\\
&\notag\quad\in C\B_\infty^{-\frac{1}{2}-\kappa}\times \L^{\frac{1}{2}-\kappa,\frac{1}{4}-\frac{1}{2}\kappa}\times (C\B_\infty^{-\kappa})^2\times (C\B_\infty^{-1-\kappa})^2\\
&\notag\quad\quad\times (C\B_\infty^{1-\kappa})^2\times (C\B_\infty^{-\kappa})^4\times (C\B_\infty^{-\frac{1}{2}-\kappa})^2
\end{align}
which satisfies $\L_\mu\ICD=\CD$ and $\L_\mu\ICC=\CC$ a {\em driving vector} of the system \eqref{solution:CGLsystem}. Let $\mathcal{X}_{\text{\rm CGL}}^\kappa$ the set of all driving vectors. For $\mathbb{X}\in\mathcal{X}_{\text{\rm CGL}}^\kappa$ and $T>0$, we define
\begin{align*}
\$\mathbb{X}\$_{\kappa,T}&=\|\IX\|_{C_T\B_\infty^{-\frac{1}{2}-\kappa}}+\|\ICCD\|_{\L_T^{\frac{1}{2}-\kappa,\frac{1}{4}-\frac{1}{2}\kappa}}+\|\CICCD\|_{C_T\B_\infty^{-\kappa}}+\|\DICCD\|_{C_T\B_\infty^{-\kappa}}\\
&\quad+\|\CD\|_{C_T\B_\infty^{-1-\kappa}}+\|\CC\|_{C_T\B_\infty^{-1-\kappa}}+\|\ICD\|_{C_T\B_\infty^{1-\kappa}}+\|\ICC\|_{C_T\B_\infty^{1-\kappa}}\\
&\quad+\|\CDICD\|_{C_T\B_\infty^{-\kappa}}+\|\CDICC\|_{C_T\B_\infty^{-\kappa}}+\|\CCJDC\|_{C_T\B_\infty^{-\kappa}}+\|\CCJDD\|_{C_T\B_\infty^{-\kappa}}\\
&\quad+\|\CDICCD\|_{C_T\B_\infty^{-\frac{1}{2}-\kappa}}+\|\CCJDDC\|_{C_T\B_\infty^{-\frac{1}{2}-\kappa}}.
\end{align*}
\end{defi}

We define the solutions of the system \eqref{solution:CGLsystem}.

\begin{defi}
For $T>0$, we call the pair $(v,w)$ of distribution-valued functions on the time interval $[0,T]$ which satisfies
\begin{equation}\label{solution:mildsolutionvw}
\begin{aligned}
v(t)&=e^{tL_\mu}v_0+\int_0^te^{(t-s)L_\mu}\{F(v,w)-cv\}(s)ds,\\
w(t)&=e^{tL_\mu}w_0+\int_0^te^{(t-s)L_\mu}\{G(v,w)+cv\}(s)ds,
\end{aligned}
\end{equation}
where $L_\mu=(i+\mu)\Delta-1$, the {\em solution} of the system \eqref{solution:CGLsystem} on $[0,T]$ with initial values $(v_0,w_0)$.
\end{defi}

\subsection{Local well-posedness}

We give the local well-posedness result of the system \eqref{solution:CGLsystem} in the space
$$
\mathcal{D}_T^{\kappa,\kappa'}=\L_T^{\frac56-\kappa',1-\kappa',1-\frac{\kappa'}2}\times\L_T^{1-\kappa'+\kappa,\frac32-2\kappa',1-\kappa'},
$$
where $0<\kappa<\kappa'<\frac1{18}$, for a short time $T$ depending on $(v_0,w_0)$ and $\mathbb{X}$. We omit the proof here. For details, see \cite[Section~4]{HIN}.

First we give the estimate of the commutator $\com(v,w)$.

\begin{lemm}[{\cite[Lemma~4.21]{HIN}}]\label{local:estimate of com}
Let $\hat{v}$ be the mild solution of
$$
\L_\mu\hat{v}=F(v,w)-cv
$$
with initial value $\hat{v}(0,\cdot)=v_0$. We define
$$
\com(v,w):=\hat{v}+\nu\{2(v+w-\nu\ICCD)\pl\ICD+(\conj{v}+\conj{w}-\conj{\nu}\JDDC)\pl\ICC\}.
$$
For every $T>0$, $p\in[1,\infty]$ and $\alpha<1+\kappa'$, we have the estimate
\begin{align*}
&\|\com(v,w)(t)\|_{\B_p^{1+\kappa'}}\\
&\lesssim1+t^{-\frac{1+\kappa'-\alpha}{2}}\|v_0\|_{\B_p^\alpha}+t^{-\kappa'}(1+\|v(t)\|_{L^p}+\|w(t)\|_{L^p})\\
&\quad+\int_0^t(t-s)^{-\frac{3+2\kappa}{4}}\|v(s)\|_{\B_p^{\frac{1}{2}+\kappa'}}ds+\int_0^t(t-s)^{-\frac{1+2\kappa'}{2}}\|w(s)\|_{\B_p^{1+2\kappa'}}ds\\
&\quad+\int_0^t(t-s)^{-1-\frac{\kappa+\kappa'}{2}}(\|\delta_{st}v\|_{L^p}+\|\delta_{st}w\|_{L^p})ds,
\end{align*}
uniformly over $t\in[0,T]$, where $\delta_{st}$ is the difference operator $\delta_{st}f:=f(t)-f(s)$. Here the implicit proportionality constant depends only on $\mu,\nu,c,\kappa,\kappa',p,\alpha,T$ and $\$\mathbb{X}\$_{\kappa,T}$.
\end{lemm}

We can obtain the local existence of the solution by a standard fixed point argument. The uniqueness and the continuity on initial values and drivers are obtained by standard PDE arguments.

\begin{theo}[{\cite[Theorem 4.26]{HIN}}]\label{local:with singular initial condition}
For every $(v_0,w_0)\in\B_\infty^{-\frac{2}{3}+\kappa'}\times\B_\infty^{-\frac{1}{2}-2\kappa}$ and $\mathbb{X}\in\mathcal{X}_{\text{\rm CGL}}^\kappa$, there exists $T_*\in(0,1]$ continuously depending on $(v_0,w_0,\mathbb{X})$ such that the system \eqref{solution:CGLsystem} has a unique solution $(v,w)\in\mathcal{D}_{T_*}^{\kappa,\kappa'}$ and this solution satisfies
\begin{align*}
\|(v,w)\|_{\mathcal{D}_{T_*}^{\kappa,\kappa'}}\lesssim1+\|v_0\|_{\B_\infty^{-\frac23+\kappa'}}+\|w_0\|_{\B_\infty^{-\frac12-2\kappa}},
\end{align*}
where the implicit constant depends only on $\mu,\nu,\lambda,c,\kappa,\kappa'$ and $\$\mathbb{X}\$_{\kappa,1}$.

Let $T_{\text{\rm sur}}\in(0,\infty]$ be the supremum of times $T$ such that the system \eqref{solution:CGLsystem} has a unique solution $(v,w)\in\mathcal{D}_T^{\kappa,\kappa'}$. If $T_{\text{\rm sur}}<\infty$, then we have
$$
\lim_{T\uparrow T_{\text{\rm sur}}}(\|v\|_{C_T\B_\infty^{-\frac23+\kappa'}}+\|w\|_{C_T\B_\infty^{-\frac12-2\kappa}})=\infty.
$$
Furthermore, this survival time $T_{\text{\rm sur}}$ is lower semicontinuous with respect to $(v_0,w_0,\mathbb{X})$, and if a sequence $(v_0^{(\epsilon)},w_0^{(\epsilon)},\mathbb{X}^{(\epsilon)})$ converge to $(v_0,w_0,\mathbb{X})$ as $\epsilon\downarrow0$, then for the corresponding solutions $(v^{(\epsilon)},w^{(\epsilon)})$ and $(v,w)$, respectively, we have
$$
\lim_{\epsilon\downarrow0}\|(v^{(\epsilon)},w^{(\epsilon)})-(v,w)\|_{\mathcal{D}_T^{\kappa,\kappa'}}=0
$$
for every $T<T_{\text{\rm sur}}$.
\end{theo}

\begin{rem}\label{local:with smooth initial condition}
If $(v_0,w_0)\in\B_\infty^{1-\kappa'}\times\B_\infty^{\frac{3}{2}-2\kappa'}$, then we can obtain the local well-posedness on the space $\L_T^{1-\kappa',1-\frac{\kappa'}{2}}\times\L_T^{\frac{3}{2}-\kappa',1-\kappa'}$ without explosions at $t=0$ by a similar argument.
\end{rem}

\subsection{Renormalization of the stochastic CGL equation}

We briefly explain the relation between the deterministic system \eqref{solution:CGLsystem} and the renormalized stochastic CGL equation \eqref{introduction:renormalized CGL}. For details, see \cite[Section~4.5]{HIN}.

As stated in Section \ref{section:introduction}, we replace the space-time white noise $\xi$ by a smeared noise $\xi^\epsilon$ which is white in $t$ but smooth in $x$. Since the stationary solution $\IX^\epsilon$ of $\L_\mu\IX^\epsilon=\xi^\epsilon$ is also smooth in $x$, we can define all products appeared in Section \ref{subsection:Definition of the solution} in usual sense. However, in order to define the convergent driving vectors $\mathbb{X}^\epsilon$ as $\epsilon\downarrow0$, we need to introduce the renormalizations of the products.

\begin{theo}[{\cite[Theorem~5.9]{HIN}}]
There exist constants $C_i^\epsilon$ ($i=1,2,3$) such that, if we define $\mathbb{X}^\epsilon$ as in Section \ref{subsection:Definition of the solution} with the additional conditions
\begin{align*}
\CD^\epsilon&=\IX^\epsilon\JY^\epsilon-C_1^\epsilon,\\
\CCD^\epsilon&=(\IX^\epsilon)^2\JY^\epsilon-2C_1^\epsilon\IX^\epsilon,\\
\CCJDD^\epsilon&=\JDD^\epsilon\rs\CC^\epsilon-2C_2^\epsilon,\\
\CDICD^\epsilon&=\ICD^\epsilon\rs\CD^\epsilon-C_3^\epsilon,\\
\CCJDDC^\epsilon&=\JDDC^\epsilon\rs\CC^\epsilon-2C_2^\epsilon\IX^\epsilon,\\
\CDICCD^\epsilon&=\ICCD^\epsilon\rs\CD^\epsilon-2C_3^\epsilon\IX^\epsilon,
\end{align*}
then there exists an $\mathcal{X}_{\text{\rm CGL}}^\kappa$-valued random variable $\mathbb{X}$ which is independent of the choice of $\eta$, and such that
$$
\Ex\$\mathbb{X}^\epsilon-\mathbb{X}\$_{\kappa,T}^p=0
$$
for every $T>0$ and $p\in[1,\infty)$. Furthermore, for the solution $(v^\epsilon,w^\epsilon)$ of the system \eqref{solution:CGLsystem} with respect to the random variable $\mathbb{X}^\epsilon$, the process $u^\epsilon=\IX^\epsilon-\nu\ICCD^\epsilon+v^\epsilon+w^\epsilon$ is a mild solution of the renormalized equation \eqref{introduction:renormalized CGL} with
$$
C^\epsilon=2C_1^\epsilon-2\conj{\nu}C_2^\epsilon-4\nu C_3^\epsilon.
$$
\end{theo}

\begin{cor}\label{localexistence:uep to u}
For every $u_0\in\B_\infty^{-\frac23+\kappa'}$, there exists a process $u$ which is independent of the choice of $\eta$, and such that the solution $u^\epsilon$ of the renormalized equation \eqref{introduction:renormalized CGL} with initial value $u_0$ satisfies
$$
\lim_{\epsilon\downarrow0}\|u^\epsilon-u\|_{C_T\B_\infty^{-\frac23+\kappa'}}=0
$$
in probability for every $T<T_{\text{\rm sur}}$, where $T_{\text{\rm sur}}$ is the survival time with respect to the driving vector $\mathbb{X}^\epsilon$ and initial values
$$
v_0^\epsilon=u_0^\epsilon-\IX^\epsilon(0)+\nu\ICCD^\epsilon(0),\quad w_0^\epsilon=0.
$$
\end{cor}

\subsection{A priori estimate of $(v,w)$}

From the above arguments, it is sufficient to show the following theorem in order to prove Theorem \ref{introduction:main theorem}.

\begin{theo}\label{localexistence+:main theorem}
Let $\mu>\frac{1}{2\sqrt{2}}$. Choose sufficiently small $0<\kappa<\kappa'$ depending on $\mu$. For every $T>0$ and $\mathbb{X}\in\mathcal{X}_{\text{\rm CGL}}^\kappa$, there exists sufficiently large $c>0$ depending only on $\mu,\nu,\lambda,\kappa,\kappa',T$ and $\$\mathbb{X}\$_{\kappa,T}$, such that, any solution $(v,w)$ of the system \eqref{solution:CGLsystem} on $[0,T]$ with initial value $(v_0,w_0)\in\B_\infty^{-\frac23+\kappa'}\times\B_\infty^{-\frac12-2\kappa}$ satisfies
\begin{align}\label{localexistence+:apriori in infty}
\|v\|_{C_T\B_\infty^{-\frac23+\kappa'}}+\|w\|_{C_T\B_\infty^{-\frac12-2\kappa}}\le C
\end{align}
for some finite constant $C>0$ depending only on $\mu,\nu,\lambda,c,\kappa,\kappa',T,\$\mathbb{X}\$_{\kappa,T},\|v_0\|_{\B_\infty^{-\frac23+\kappa'}}$ and $\|w_0\|_{\B_\infty^{-\frac12-2\kappa}}$.
\end{theo}

Although we consider the system \eqref{solution:CGLsystem} with different $c>0$ for each fixed final time $T>0$, the renormalized equation \eqref{introduction:renormalized CGL} is irrelevant to the choice of $c$. Theorem \ref{localexistence+:main theorem} implies that the solution $u=\IX-\nu\ICCD+v+w$ does not explode in the space $\B_\infty^{-\frac23+\kappa'}$ until every fixed $T>0$, so that the result of Corollary \ref{localexistence:uep to u} holds for all $T>0$.

We show Theorem \ref{localexistence+:main theorem} in the rest of this \paper\ by the method explained in Section \ref{section:introduction}. Our goal is the a priori $L^\infty[0,T]$ estimate
\begin{align}\label{localexistence:a priori in besov}
\|v\|_{C_T\B_{2p+2}^{\frac12+\kappa'}}+\|w\|_{C_T\B_{\frac{2p+2}3}^{\frac32-2\kappa'}}\le C'<\infty
\end{align}
for $p>\frac32$, instead of the estimate \eqref{localexistence+:apriori in infty}. If the estimate \eqref{localexistence:a priori in besov} is true, then the Besov embeddings
\begin{align*}
&\B_\infty^{-\frac23+\kappa'}\supset\B_{2p+2}^{-\frac23+\frac3{2p+2}+\kappa'}\supset\B_{2p+2}^{\frac12+\kappa'},\\
&\B_\infty^{-\frac12-2\kappa}\supset\B_{\frac{2p+2}3}^{-\frac12+\frac9{2p+2}-2\kappa}\supset\B_{\frac{2p+2}3}^{\frac32-2\kappa'}
\end{align*}
imply the a priori estimate \eqref{localexistence+:apriori in infty}. Additionally, since we already have
$$
\|v(T_*)\|_{\B_\infty^{1-\kappa'}}+\|w(T_*)\|_{\B_\infty^{\frac32-2\kappa'}}\lesssim1+\|v_0\|_{\B_\infty^{-\frac23+\kappa'}}+\|w_0\|_{\B_\infty^{-\frac12-2\kappa}}
$$
from Theorem \ref{local:with singular initial condition}, we assume that {\bf the initial value} $(v_0,w_0)$ {\bf belongs to} $\B_\infty^{1-\kappa'}\times\B_\infty^{\frac32-2\kappa'}$ {\bf in what follows} without loss of generality, by starting the argument from the time $T_*$.

From now on, we fix $T>0$ and $\mathbb{X}\in\mathcal{X}_{\text{\rm CGL}}^\kappa$. In the inequalities shown below, we do not remark the dependences of the proportionality constants on the parameters $\mu,\nu,\lambda,\kappa,\kappa',p,T$ and $\$\mathbb{X}\$_{\kappa,T}$.

\section{A priori estimate of $v$}\label{section:apriori_v}

In this section, we will show that the Besov norms of $v$ and $\com(v,w)$ are controlled by the $L^p$ norm of $w$. The following theorem is obtained by the same arguments as \cite[Theorem~3.1]{MW16}.

\begin{theo}
Let $p\in[1,\infty)$ and $c>0$. Then for every $0\le s\le t\le T$,
\begin{align}
\label{v:estimate of v in L^p}\|v(t)\|_{L^p}&\lesssim e^{-ct}\|v_0\|_{L^p}+\int_0^te^{-c(t-s)}(t-s)^{-\frac{1+\kappa'}{2}}(1+\|w(s)\|_{L^p})ds,\\
\label{v:estimate of v in Besov}\|v(t)\|_{\B_p^{\frac{1}{2}+\kappa'}}&\lesssim\|v_0\|_{\B_p^{\frac{1}{2}+\kappa'}}+\int_0^t(t-s)^{-\frac{3}{4}-\kappa'}(1+\|w(s)\|_{L^p})ds,\\
\label{v:estimate of dv in L^p}\|\delta_{st}v\|_{L^p}&\lesssim(t-s)^{\frac{1+2\kappa}{4}}\|v(s)\|_{\B_p^{\frac{1}{2}+\kappa'}}+\int_s^t(t-r)^{-\frac{1+\kappa'}{2}}(1+\|w(r)\|_{L^p})dr,
\end{align}
where the implicit constants do not depend on $c>0$.
\end{theo}

\begin{proof}
The definition \eqref{solution:mildsolutionvw} of the solution $v$ is equivalent to
$$
v(t)=e^{tL_\mu^c}v_0+\int_0^te^{(t-s)L_\mu^c}F(v,w)(s)ds,
$$
where $L_\mu^c=(i+\mu)\Delta-(c+1)$. For every $\alpha\in(0,1-\kappa)$, we have
\begin{align}\label{v:estimate of v in Besov proof}
\|e^{tL_\mu^c}v_0\|_{\B_p^\alpha}\lesssim e^{-(c+1)t}\|v_0\|_{\B_p^\alpha}
\end{align}
and
\begin{align*}
&\left\|\int_0^te^{(t-s)L_\mu^c}F(v,w)(s)ds\right\|_{\B_p^\alpha}\\
&\lesssim\int_0^te^{-(c+1)(t-s)}(t-s)^{-\frac{\alpha+1+\kappa}2}\|F(v,w)(s)\|_{\B_p^{-1-\kappa}}ds\\
&\lesssim\int_0^te^{-(c+1)(t-s)}(t-s)^{-\frac{\alpha+1+\kappa}2}(1+\|v(s)\|_{\B_p^\alpha}+\|w(s)\|_{L^p})ds.
\end{align*}
Hence by \cite[Lemma~3.4]{MW16}, we have
\begin{align*}
\|v(t)\|_{\B_p^\alpha}\lesssim e^{-\underline{c}t}\|v_0\|_{\B_p^\alpha}+\int_0^te^{-\underline{c}(t-s)}(t-s)^{-\frac{\alpha+1+\kappa}2}(1+\|w(s)\|_{L^p})ds,
\end{align*}
where $\underline{c}=c-[\Gamma(\frac{1-\kappa-\alpha}2)]^{\frac2{1-\kappa-\alpha}}$. Here we can replace $\underline{c}$ by $c$ again because we ignore the factor depending only on $\kappa,\alpha$ and $T$. The second assertion \eqref{v:estimate of v in Besov} is obtained by setting $\alpha=\frac12+\kappa'$ and using $e^{-ct}\le1$. The first assertion \eqref{v:estimate of v in L^p} is obtained by setting $\alpha=\kappa'-\kappa$ and using
$$
\|e^{tL_\mu^c}v_0\|_{L^p}\lesssim e^{-(c+1)t}\|v_0\|_{L^p}
$$
instead of \eqref{v:estimate of v in Besov proof}.

In order to show the third assertion \eqref{v:estimate of dv in L^p}, we need to estimate
\begin{align*}
\delta_{st}v=(e^{(t-s)L_\mu^c}-1)v(s)+\int_s^te^{(t-r)L_\mu^c}F(v,w)(r)dr.
\end{align*}
For the first term, we have
\begin{align*}
\|(e^{(t-s)L_\mu^c}-1)v(s)\|_{L^p}\lesssim\|(e^{(t-s)L_\mu^c}-1)v(s)\|_{\B_p^{\kappa'-\kappa}}\lesssim(t-s)^{\frac{1+2\kappa}{4}}\|v(s)\|_{\B_p^{\frac12+\kappa'}}.
\end{align*}
For the second term, we have
\begin{align*}
\left\|\int_s^te^{(t-r)L_\mu^c}F(v,w)(r)dr\right\|_{\B_p^{\kappa'-\kappa}}
&\lesssim\int_s^t(t-r)^{-\frac{1+\kappa'}2}\|F(v,w)(r)\|_{\B_p^{-1-\kappa}}dr\\
&\lesssim\int_s^t(t-r)^{-\frac{1+\kappa'}2}(1+\|v(r)\|_{\B_p^{\frac12+\kappa'}}+\|w(r)\|_{L^p})dr.
\end{align*}
We can bound the part involving $\|v(r)\|_{\B_p^{\frac12+\kappa'}}$ by
\begin{align*}
&\int_s^t(t-r)^{-\frac{1+\kappa'}2}\|v(r)\|_{\B_p^{\frac12+\kappa'}}dr\\
&\lesssim\int_s^t(t-r)^{-\frac{1+\kappa'}2}\|v(s)\|_{\B_p^{\frac12+\kappa'}}dr+\int_s^t(t-r)^{-\frac{1+\kappa'}2}\int_0^r(r-\tau)^{-\frac34-\kappa'}(1+\|w(\tau)\|_{L^p})d\tau\\
&\lesssim(t-s)^{\frac{1-\kappa'}2}\|v(s)\|_{\B_p^{\frac12+\kappa'}}+\int_s^t\left(\int_\tau^t(t-r)^{-\frac{1+\kappa'}2}(r-\tau)^{-\frac34-\kappa'}dr\right)(1+\|w(\tau)\|_{L^p})d\tau\\
&\lesssim(t-s)^{\frac{1-\kappa'}2}\|v(s)\|_{\B_p^{\frac12+\kappa'}}+\int_s^t(t-\tau)^{-\frac{1+6\kappa'}{4}}(1+\|w(\tau)\|_{L^p})d\tau.
\end{align*}
In this \paper, we repeatedly use the exchange of the order of integration like above.
\end{proof}

As an application, we can control $\com(v,w)$ by $w$.

\begin{cor}
Let $p\in[1,\infty)$ and $c>0$. Then for every $t\in[0,T]$,
\begin{align}\label{v:estimate of com in Besov}
\|\com(v,w)(t)\|_{\B_p^{1+\kappa'}}&\lesssim1+t^{-\frac{1}{4}}\|v_0\|_{\B_p^{\frac{1}{2}+\kappa'}}+t^{-\kappa'}(1+\|w(t)\|_{L^p})\\\notag
&\quad+t^{-\kappa'}\int_0^t(t-s)^{-\frac{1+\kappa'}{2}}(1+\|w(s)\|_{L^p})ds\\\notag
&\quad+\int_0^t(t-s)^{-\frac{1+3\kappa'}{2}}\|w(s)\|_{\B_p^{1+2\kappa'}}ds\\\notag
&\quad+\int_0^t(t-s)^{-1-\frac{\kappa+\kappa'}{2}}\|\delta_{st}w\|_{L^p}ds,
\end{align}
where the implicit constant depends on $c$.
\end{cor}

\begin{proof}
We use the estimate in Lemma \ref{local:estimate of com}, setting $\alpha=\frac{1}{2}+\kappa'$. We need to control the terms
\[t^{-\kappa'}\|v(t)\|_{L^p},\quad\int_0^t(t-s)^{-\frac{3+2\kappa}{4}}\|v(s)\|_{\B_p^{\frac{1}{2}+\kappa'}}ds,\quad\int_0^t(t-s)^{-1-\frac{\kappa+\kappa'}{2}}\|\delta_{st}v\|_{L^p}ds\]
by $w$. For the first term, we use (\ref{v:estimate of v in L^p}) and have
\begin{align*}
t^{-\kappa'}\|v(t)\|_{L^p}\lesssim
t^{-\kappa'}\|v_0\|_{L^p}+t^{-\kappa'}\int_0^t(t-s)^{-\frac{1+\kappa'}{2}}(1+\|w(s)\|_{L^p})ds.
\end{align*}
For the second term, from (\ref{v:estimate of v in Besov})
\begin{align}\label{v:second term}
&\int_0^t(t-s)^{-\frac{3+2\kappa}{4}}\|v(s)\|_{\B_p^{\frac{1}{2}+\kappa'}}ds\\\notag
&\lesssim\|v_0\|_{\B_p^{\frac{1}{2}+\kappa'}}+\int_0^t(t-s)^{-\frac{3+2\kappa}{4}}\int_0^s(s-r)^{-\frac{3}{4}-\kappa'}(1+\|w(r)\|_{L^p})dr\,ds\\\notag
&=\|v_0\|_{\B_p^{\frac{1}{2}+\kappa'}}+\int_0^t\left(\int_r^t(t-s)^{-\frac{3+2\kappa'}{4}}(s-r)^{-\frac{3+4\kappa'}{4}}ds\right)(1+\|w(r)\|_{L^p})dr\\\notag
&\lesssim\|v_0\|_{\B_p^{\frac{1}{2}+\kappa'}}+\int_0^t(t-r)^{-\frac{1+3\kappa'}{2}}(1+\|w(r)\|_{L^p})dr.
\end{align}
For the third term, from (\ref{v:estimate of dv in L^p})
\begin{align*}
\int_0^t(t-s)^{-1-\frac{\kappa+\kappa'}{2}}\|\delta_{st}v\|_{L^p}ds
&\lesssim\int_0^t(t-s)^{-\frac{3+2\kappa'}{4}}\|v(s)\|_{\B_p^{\frac{1}{2}+\kappa'}}ds\\
&\quad+\int_0^t(t-s)^{-1-\kappa'}\int_s^t(t-r)^{-\frac{1+\kappa'}{2}}(1+\|w(r)\|_{L^p})dr\,ds
\end{align*}
Here the first integral is bounded by (\ref{v:second term}) again. The second integral is computed by
\begin{align*}
&\int_0^t\int_0^r(t-s)^{-1-\kappa'}ds\,(t-r)^{-\frac{1+\kappa'}{2}}(1+\|w(r)\|_{L^p})dr\\
&\lesssim\int_0^t(t-r)^{-\frac{1+3\kappa'}{2}}(1+\|w(r)\|_{L^p})dr.
\end{align*}
These complete the proof.
\end{proof}

\section{A priori estimate of $w$}\label{section:apriori_w}

The goal of this section is to show the following theorem.

\begin{theo}\label{w:control L^p by Besov}
Let $p\in(1,5\wedge\{1+\mu(\mu+\sqrt{1+\mu^2})\})$ and assume $\frac{5}{2}\kappa'\le\frac{3}{4}-\frac{3}{2p+2}$. For sufficiently large $c$ depending on $\mu,\nu,\lambda,\kappa,\kappa',p,T$ and $\$\mathbb{X}\$_{\kappa,T}$, we have
\[\|w(t)\|_{L^{2p}}^{2p}+\int_0^t\|w(s)\|_{L^{2p+2}}^{2p+2}ds
\lesssim1+\|v_0\|_{\B_{2p+2}^{\frac{1}{2}+\kappa'}}^{2p+2}+\|w_0\|_{L^{2p}}^{2p}+\int_0^t\|w(s)\|_{\B_{\frac{2p+2}{3}}^{1+2\kappa'}}^{\frac{2p+2}{3}}ds,\]
where the implicit constant depends only on $\mu,\nu,\lambda,c,\kappa,\kappa',p,T$ and $\$\mathbb{X}\$_{\kappa,T}$.
\end{theo}

We start from the following $L^{2p}$ inequality. See also \cite[Section~4]{DGL}.

\begin{prop}\label{w:L^2p control of w from PDE}
Let $1<p<1+\mu(\mu+\sqrt{1+\mu^2})$. For every $\delta>0$ such that
\begin{align}\label{w:assumption of p and mu}
\frac{p-1}{\mu(\mu+\sqrt{1+\mu^2})}<1-\delta,
\end{align}
we have the following inequality.
\begin{multline}\label{w:L^p inequality}
\frac{1}{2p}(\|w(t)\|_{L^{2p}}^{2p}-\|w_0\|_{L^{2p}}^{2p})+\delta\mu\int_0^t\||\nabla w|^2|w|^{2p-2}(s)\|_{L^1}ds\\
{}+\Re\nu\int_0^t\|w(s)\|_{L^{2p+2}}^{2p+2}\le\int_0^t\langle|w|^{2p-2},\Re(\conj{w}G_c')\rangle(s)ds.
\end{multline}
Here $G_c'(v,w)=G(v,w)+cv+\nu w^2\bar{w}$.
\end{prop}

\begin{proof}
We compute the derivative of $\|w(t)\|_{L^{2p}}^{2p}$ at formal level. For every $p>1$,
\begin{align}\label{w:derivative of L^2p}
&\frac{d}{dt}\|w(t)\|_{L^{2p}}^{2p}
=\frac{d}{dt}\int_{\T^3}(w\conj{w})^{p}dx
=p\int_{\T^3}(w\conj{w})^{p-1}(w\partial_t\conj{w}+\conj{w}\partial_tw)dx\\\notag
&=p\int_{\T^3}(w\conj{w})^{p-1}\{(-i+\mu)w\Delta\conj{w}+(i+\mu)\conj{w}\Delta w\}dx+p\int_{\T^3}(w\conj{w})^{p-1}(w\conj{G_c}+\conj{w}G_c)dx\\\notag
&=-p\left[(-i+\mu)\int_{\T^3}\nabla\{(w\conj{w})^{p-1}w\}\cdot\nabla\conj{w}dx+(i+\mu)\int_{\T^3}\nabla\{(w\conj{w})^{p-1}\conj{w}\}\cdot\nabla wdx\right]\\\notag
&\quad-2p\Re\nu\|w(t)\|_{L^{2p+2}}^{2p+2}+p\int_{\T^3}|w|^{2p-2}(w\conj{G_c'}+\conj{w}G_c')dx,
\end{align}
where $G_c(v,w)=G(v,w)+cv$.

We can justify the above computations as follows. First, since $w(t)$ is not differentiable in $t$, we should interpret \eqref{w:derivative of L^2p} as the integration equality
\[\|w(t)\|_{L^{2p}}^{2p}-\|w(0)\|_{L^{2p}}^{2p}=\int_0^t\cdots ds.\]
Then $\partial_tw$ and $\partial_t\conj{w}$ are defined by Young integrals:
\begin{align}\label{w:young integral}
p\int_0^t\langle(w\conj{w})^{p-1}w,\partial_s\conj{w}\rangle+p\int_0^t\langle(w\conj{w})^{p-1}\conj{w},\partial_sw\rangle.
\end{align}
We can see that $w\in C_T^\delta L^\infty$ for $\delta<\frac{3}{4}-\kappa'$ by the definition of the solution space $\mathcal{D}_T^{\kappa,\kappa'}$. (Since $w_0\in\B_\infty^{\frac32-2\kappa'}$ now, $w$ belongs to $\L_T^{\frac32-2\kappa',1-\kappa'}$ rather than $\L_T^{1-\kappa'+\kappa,\frac32-2\kappa',1-\kappa'}$.) Since the function $w\mapsto|w|^{2p-2}w$ is locally Lipschitz continuous because $(2p-2)+1>1$, the above Young integrals are well-defined. The last equality in \eqref{w:derivative of L^2p} is justified by classical PDE theory. By a similar argument to \cite[Proposition~6.7]{MW15}, the mild solution $w$ is also a weak solution, in the sense that for every $\varphi\in \B_\infty^1$,
\begin{align}\label{w:mild sol is weak sol}
\langle w(t),\varphi\rangle-\langle w_0,\varphi\rangle=-(i+\mu)\int_0^t\langle\nabla w(s),\nabla\varphi\rangle ds+\int_0^t\langle G_c(s),\varphi\rangle ds.
\end{align}
Let $\varphi_s=(\conj{w}|w|^{2p-2})(s)$ for $s\in[0,t]$. Since $\nabla w\in C_TL^\infty$ and
\begin{align*}
\nabla\{w(w\conj{w})^{p-1}\}=p(w\conj{w})^{p-1}\nabla w+(p-1)(w\conj{w})^{p-2}w^2\nabla\conj{w}\in C_TL^\infty,
\end{align*}
we have $\varphi_s\in\B_\infty^1$ by Proposition \ref{para:sobolev into besov}. Hence it is allowed to insert $\varphi=\varphi_s$ into \eqref{w:mild sol is weak sol}. We take a partition $\{0=t_0<\dots<t_N=t\}$ of $[0,t]$ and consider the sum
\begin{align*}
&\sum_{i=0}^{N-1}\langle w_{t_{i+1}}-w_{t_i},\varphi_{t_i}\rangle\\
&=-(i+\mu)\sum_{i=0}^{N-1}\int_{t_i}^{t_{i+1}}\langle\nabla w(s),\nabla\varphi_{t_i}\rangle ds+\sum_{i=0}^{N-1}\int_{t_i}^{t_{i+1}}\langle G_c(s),\varphi_{t_i}\rangle ds.
\end{align*}
As $\sup_{i}|t_{i+1}-t_i|\to0$, the left hand side becomes Young integral as \eqref{w:young integral}. The right hand side also converges to Riemann integrals
\begin{align*}
-(i+\mu)\int_0^t\langle\nabla w,\nabla\{\conj{w}(w\conj{w})^{p-1}\}\rangle(s) ds+\int_0^t\langle G_c,\conj{w}(w\conj{w})^{p-1}\rangle(s)ds.
\end{align*}

Now we return to the first term of the last part of \eqref{w:derivative of L^2p}. Since
\begin{align*}
\langle\nabla\{(w\conj{w})^{p-1}w\},\nabla\conj{w}\rangle
&=\langle(w\conj{w})^{p-1},|\nabla\conj{w}|^2\rangle+(p-1)\langle(w\conj{w})^{p-2},w\nabla\conj{w}\cdot\nabla(w\conj{w})\rangle\\
&=p\langle|w|^{2p-2},|\nabla w|^2\rangle+(p-1)\langle|w|^{2p-4},w^2(\nabla\conj{w})^2\rangle
\end{align*}
we have
\begin{align*}
&(-i+\mu)\langle\nabla\{(w\conj{w})^{p-1}w\},\nabla\conj{w}\rangle+(i+\mu)\langle\nabla\{(w\conj{w})^{p-1}\conj{w}\},\nabla w\rangle\\
&=2p\mu\langle|w|^{2p-2},|\nabla w|^2\rangle+(p-1)\langle|w|^{2p-4},(-i+\mu)w^2(\nabla\conj{w})^2+(i+\mu)\conj{w}^2(\nabla w)^2\rangle\\
&=2p\mu\langle |w|^{2p-2},|\nabla w|^2\rangle+(p-1)\mu\langle|w|^{2p-4},(w\nabla\conj{w}-\conj{w}\nabla w)^2+2|w|^2|\nabla w|^2\rangle\\
&\quad-(p-1)i\langle|w|^{2p-4},(w\nabla\conj{w}-\conj{w}\nabla w)\cdot(w\nabla\conj{w}+\conj{w}\nabla w)\rangle\\
&=2(2p-1)\mu\langle |w|^{2p-2},|\nabla w|^2\rangle-(p-1)\mu\langle|w|^{2p-4},|w\nabla\conj{w}-\conj{w}\nabla w|^2\rangle\\
&\quad-(p-1)\langle|w|^{2p-4},i(w\nabla\conj{w}-\conj{w}\nabla w)\cdot\nabla|w|^2\rangle.
\end{align*}
Let $\delta\in(0,1]$ and move the term $-2p\delta\mu\langle |w|^{2p-2},|\nabla w|^2\rangle$ into the left hand side. Then the quantity
\begin{align*}
&-p[2(2p-1-\delta)\mu\langle |w|^{2p-2},|\nabla w|^2\rangle-(p-1)\mu\langle|w|^{2p-4},|w\nabla\conj{w}-\conj{w}\nabla w|^2\rangle\\
&\quad-(p-1)\langle|w|^{2p-4},i(w\nabla\conj{w}-\conj{w}\nabla w)\cdot\nabla|w|^2\rangle]
\end{align*}
remains. By using the identity $4|w|^2|\nabla w|^2=(\nabla|w|^2)^2+|w\nabla\conj{w}-\conj{w}\nabla w|^2$, the above value turns into $-p\langle|w|^{2p-4},f\rangle$, where
\begin{align*}
f&=\left(p-\frac{1}{2}-\frac{\delta}{2}\right)\mu(\nabla|w|^2)^2-(p-1)\nabla|w|^2\cdot i(w\nabla\conj{w}-\conj{w}\nabla w)\\
&\quad+\left(\frac{1}{2}-\frac{\delta}{2}\right)\mu|w\nabla\conj{w}-\conj{w}\nabla w|^2.
\end{align*}
This quadratic form is nonnegative if the matrix
$$
\begin{pmatrix}
\left(p-\frac{1}{2}-\frac{\delta}{2}\right)\mu&-\frac{1}{2}(p-1)\\
-\frac{1}{2}(p-1)&\left(\frac{1}{2}-\frac{\delta}{2}\right)\mu
\end{pmatrix}
$$
is nonnegative definite ($\Leftrightarrow$ has nonnegative trace and nonnegative determinant), i.e. the condition \eqref{w:assumption of p and mu} holds.
\end{proof}

The right hand side of (\ref{w:L^p inequality}) is written as
\begin{align*}
\int_0^t\langle|w|^{2p-2},\Re(\conj{w}G_c')\rangle(s)ds
=\sum_{i=1}^8\mathcal{I}_{(i)}(t),
\end{align*}
where
\begin{align*}
\mathcal{I}_{(1)}(t)&=-\nu\int_0^t\langle|w|^{2p-2},\Re\{\conj{w}(v^2\conj{v}+v^2\conj{w}+2v\conj{v}w+2vw\conj{w}+\conj{v}w^2)\}\rangle(s)ds,\\
\mathcal{I}_{(3)}(t)&=\int_0^t\langle|w|^{2p-2},\Re\{\conj{w}(G_{(3)}+cv)\}\rangle(s)ds,\\
\mathcal{I}_{(i)}(t)&=\int_0^t\langle|w|^{2p-2},\Re(\conj{w}G_{(i)})\rangle(s)ds\quad(i\neq1,3).
\end{align*}
In Lemmas \ref{w:estimateI1}-\ref{w:estimateI7}, we will show that each of $\mathcal{I}_{(i)}$s are controlled by the following integrals.
\begin{align*}
A_t&=\int_0^ta_sds,&a_s&=1+\|w(s)\|_{L^{2p+2}}^{2p+2},\\
B_t&=\int_0^tb_sds,&b_s&=1+\||\nabla w|^2|w|^{2p-2}(s)\|_{L^1},\\
C_t&=\int_0^tc_sds,&c_s&=\|w(s)\|_{\B_{\frac{2p+2}{3}}^{1+2\kappa'}}^{\frac{2p+2}{3}}.
\end{align*}
Here we put the extra term $1$ in the definitions of $a_s$ and $b_s$ to ensure that $a_s^\alpha\le a_s^\beta$ and $b_s^\alpha\le b_s^\beta$ for $\alpha<\beta$. Our main tools are discrete Young's inequality and Jensen's inequality:
\begin{itemize}
\setlength{\itemsep}{1mm}
\item For every $\alpha_1,\dots,\alpha_N>0$ such that $\sum\alpha_i=1$ and $\epsilon>0$, there exists $C_\epsilon$ such that
\[\prod_ix_i^{\alpha_i}\le C_\epsilon x_1+\epsilon\sum_{i\neq1}x_i\]
for every $x_i\ge0$.
\item Let $f(t)$ be a nonnegative and integrable function on $[0,T]$. Then there exists a constant $C_f$ such that, for every $p>1$ and nonnegative function $g(t)$ on $[0,T]$, we have
\[\left(\int_0^Tf(t)g(t)dt\right)^p\le C_f\int_0^Tf(t)g(t)^pdt.\]
\end{itemize}
In the following lemmas, we always write $C_\epsilon$ for a large constant depending only on $\epsilon,\mu,\nu,\lambda,c,\kappa,\kappa',p,T$ and $\$\mathbb{X}\$_{\kappa,T}$.

\begin{lemm}\label{w:estimateI1}
Let $p>1$ and $\epsilon>0$. For sufficiently large $c$ depending only on $\epsilon,\mu,\nu,\lambda,\kappa,\kappa',p,T$ and $\$\mathbb{X}\$_{\kappa,T}$, we have
\[\mathcal{I}_{(1)}(t)
\le\epsilon(\|v_0\|_{L^{2p+2}}^{2p+2}+A_t).\]
\end{lemm}

\begin{proof}
By Young's inequality, we easily have
\[\mathcal{I}_{(1)}(t)\le\epsilon\int_0^t\|w(s)\|_{L^{2p+2}}^{2p+2}ds+C_\epsilon\int_0^t\|v(s)\|_{L^{2p+2}}^{2p+2}ds,\]
where the constant $C_\epsilon$ depends only on $\epsilon$ and $\nu$. From \eqref{v:estimate of v in L^p}, we have
\begin{align}\label{w:G_1:estimating v in L^p}
&\int_0^t\|v(s)\|_{L^{2p+2}}^{2p+2}ds\\\notag
&\lesssim\int_0^t\left\{e^{-cs}\|v_0\|_{L^{2p+2}}+\int_0^se^{-c(s-r)}(s-r)^{-\frac{1+\kappa'}{2}}a_r^{\frac1{2p+2}}dr\right\}^{2p+2}ds\\\notag
&\lesssim\int_0^te^{-(2p+2)cs}\|v_0\|_{L^{2p+2}}^{2p+2}+\int_0^t\int_0^se^{-c(s-r)}(s-r)^{-\frac{1+\kappa'}{2}}a_rdr\,ds\\\notag
&\lesssim \frac1c\|v_0\|_{L^{2p+2}}^{2p+2}+K(c)\int_0^ta_sds,
\end{align}
where $K(c)=\int_0^\infty e^{-cs}s^{-\frac{1+\kappa'}{2}}ds$. In the second inequality, we used Jensen's inequality. Since $\frac1c+K(c)\downarrow0$ as $c\to\infty$, we have the required estimate by choosing sufficiently large $c>0$.
\end{proof}

\begin{lemm}
For every $p>1$ and $\epsilon>0$, we have
\begin{align*}
\mathcal{I}_{(5)}(t)&
\le C_\epsilon+\epsilon(\|v_0\|_{\B_{2p+2}^{\frac{1}{2}+\kappa'}}^{2p+2}+A_t+C_t),\\
\mathcal{I}_{(8)}(t)&
\le C_\epsilon+\epsilon(\|v_0\|_{\B_{2p+2}^{\frac{1}{2}+\kappa'}}^{2p+2}+A_t+B_t+C_t).
\end{align*}
\end{lemm}

\begin{proof}
We focus on the second one since the first one is shown more easily. From Proposition \ref{para:estimates of Besov norm}, we have
\begin{align}\label{w:estimating G_8}
|\langle|w|^{2p-2},\conj{w}G_{(8)}\rangle|&
\lesssim\|w|w|^{2p-2}\|_{\B_{\frac{p+1}{p},\infty}^{\frac{1}{2}}}
\|G_{(8)}\|_{\B_{p+1,1}^{-\frac{1}{2}}}\\\notag
&\lesssim\|w|w|^{2p-2}\|_{\B_{\frac{p+1}{p},\infty}^{\frac{1}{2}}}
\|G_{(8)}\|_{\B_{p+1,\infty}^{-\frac{1}{2}+\kappa'-\kappa}}\\\notag
&\lesssim\|w|w|^{2p-2}\|_{\B_{\frac{p+1}{p}}^{\frac{1}{2}}}
(1+\|v\|_{\B_{p+1}^{\frac{1}{2}+\kappa'}}+\|w\|_{\B_{p+1}^{\frac{1}{2}+\kappa'}}).
\end{align}
We apply Proposition \ref{para:sobolev into besov} to $w|w|^{2p-2}=w(w\conj{w})^{p-1}$. Since
\begin{align*}
\nabla\{w(w\conj{w})^{p-1}\}&=p(w\conj{w})^{p-1}\nabla w+(p-1)(w\conj{w})^{p-2}w^2\nabla\conj{w}\\
&=|w|^{p-1}\nabla w\cdot p|w|^{p-1}+|w|^{p-1}\nabla\conj{w}\cdot(p-1)|w|^{p-3}w^2,
\end{align*}
by H\"older's inequality we have
\begin{align*}
\|\nabla(w|w|^{2p-2})\|_{L^{\frac{p+1}{p}}}
&\lesssim\||w|^{p-1}\nabla w\|_{L^2}\||w|^{p-1}\|_{L^{\frac{2p+2}{p-1}}}\\
&\lesssim\||w|^{2p-2}|\nabla w|^2\|_{L^1}^{\frac{1}{2}}\|w\|_{L^{2p+2}}^{p-1}=a^{\frac{p-1}{2p+2}}b^{\frac12}.
\end{align*}
Combining this with $\|w|w|^{2p-2}\|_{L^{\frac{p+1}{p}}}\lesssim\|w\|_{L^{2p+2}}^{2p-1}=a^{\frac{2p-1}{2p+2}}$, we have
\begin{align*}
\|w|w|^{2p-2}\|_{\B_{\frac{p+1}{p}}^{\frac{1}{2}}}
&\lesssim\|w|w|^{2p-2}\|_{L^{\frac{p+1}{p}}}^{\frac{1}{2}}
\|\nabla(w|w|^{2p-2})\|_{L^{\frac{p+1}{p}}}^{\frac{1}{2}}
+\|w|w|^{2p-2}\|_{L^{\frac{p+1}{p}}}\\
&\lesssim a^{\frac{1}{2}\frac{2p-1}{2p+2}+\frac{1}{2}\frac{p-1}{2p+2}}b^{\frac{1}{4}}+a^{\frac{2p-1}{2p+2}}\lesssim a^\alpha+b^\alpha,
\end{align*}
where $\alpha=\frac{2p-\frac{1}{2}}{2p+2}$.

We consider the time integral of (\ref{w:estimating G_8}). For the term involving $v$,  by Young's inequality we have
\begin{align*}
&\int_0^t(1+\|v(s)\|_{\B_{p+1}^{\frac{1}{2}+\kappa'}})(a_s^\alpha+b_s^\alpha)ds\lesssim C_\epsilon+\epsilon\int_0^t(1+\|v(s)\|_{\B_{2p+2}^{\frac{1}{2}+\kappa'}}^{2p+2})ds+\epsilon\int_0^t(a_s+b_s)ds,
\end{align*}
since $\frac{1}{2p+2}+\alpha<1$. The second term is estimated by the similar computations to those in (\ref{w:G_1:estimating v in L^p}) as follows.
\begin{align*}
\int_0^t\|v(s)\|_{\B_{2p+2}^{\frac{1}{2}+\kappa'}}^{2p+2}ds
&\lesssim\int_0^t\left\{\|v_0\|_{\B_{2p+2}^{\frac{1}{2}+\kappa'}}+\int_0^s(s-r)^{-\frac{3}{4}-\kappa'}(1+\|w(r)\|_{L^{2p+2}})dr\right\}^{2p+2}ds\\
&\lesssim\|v_0\|_{\B_{2p+2}^{\frac{1}{2}+\kappa'}}^{2p+2}+A_t.
\end{align*}
For the term involving $w$, we need the interpolation
\[\|w\|_{\B_{p+1}^{\frac{1}{2}+\kappa'}}\lesssim\|w\|_{\B_{\frac{2p+2}{3}}^{1+2\kappa'}}^{\frac{1}{2}}\|w\|_{L^{2p+2}}^{\frac{1}{2}}\le a^{\frac{1}{2}\frac{1}{2p+2}}c^{\frac{1}{2}\frac{3}{2p+2}}\lesssim a^{\frac{1}{p+1}}+c^{\frac{1}{p+1}}.\]
Since $\frac{1}{p+1}+\alpha<1$, by Young's inequality we have
\[\int_0^t(a_s^{\frac{1}{p+1}}+c_s^{\frac{1}{p+1}})(a_s^\alpha+b_s^\alpha)ds\lesssim
C_\epsilon+\epsilon\int_0^t(a_s+b_s+c_s)ds.\]
These complete the proof.
\end{proof}

\begin{lemm}\label{w:lemm:estimate of G_6}
Let $p\in(1,5)$ and assume $\frac{5}{2}\kappa'\le\frac{3}{4}-\frac{3}{2p+2}$. For every $\epsilon>0$, we have
\begin{align*}
\mathcal{I}_{(6)}(t)&
\le C_\epsilon C_t+\epsilon(\|v_0\|_{\B_{2p+2}^{\frac{1}{2}+\kappa'}}^{2p+2}+A_t).
\end{align*}
\end{lemm}

\begin{proof}
Since
\begin{align*}
|\langle|w|^{2p-2},\conj{w}G_{(6)}\rangle|\lesssim\|G_{(6)}\|_{L^{\frac{2p+2}{3}}}\|w|w|^{2p-2}\|_{L^{\frac{2p+2}{2p-1}}}\lesssim\|\com(v,w)\|_{\B_{\frac{2p+2}{3}}^{1+\kappa'}}\|w\|_{L^{2p+2}}^{2p-1},
\end{align*}
we have
\begin{align*}
\mathcal{I}_{(6)}(t)
&\lesssim\left(\int_0^t\|\com(v,w)(s)\|_{\B_{\frac{2p+2}{3}}^{1+\kappa'}}^{\frac{2p+2}{3}}ds\right)^{\frac{3}{2p+2}}\left(\int_0^ta_sds\right)^{\frac{2p-1}{2p+2}}\\
&\lesssim C_\epsilon\int_0^t\|\com(v,w)(s)\|_{\B_{\frac{2p+2}{3}}^{1+\kappa'}}^{\frac{2p+2}{3}}ds+\epsilon A_t
\end{align*}
We consider the time integral of each term in \eqref{v:estimate of com in Besov}. The first term is trivial. Integrability of the second term $t^{-\frac{1}{4}}\|v_0\|_{\B_{\frac{2p+2}3}^{\frac{1}{2}+\kappa'}}$ is easy because $\frac{1}{4}\frac{2p+2}{3}<1$ by assumption. For the third and fourth terms, because $\kappa'(p+1)<6\kappa'<1$ we have
\begin{align*}
\int_0^t\{s^{-\kappa'}(1+\|w(s)\|_{L^{\frac{2p+2}{3}}})\}^{\frac{2p+2}{3}}ds
\lesssim\left(\int_0^ts^{-\kappa'(p+1)}ds\right)^{\frac{2}{3}}
\left(\int_0^ta_s\right)^{\frac{1}{3}}
\lesssim A_t^{\frac{1}{3}},
\end{align*}
and
\begin{align*}
&\int_0^t\left\{s^{-\kappa'}\int_0^s(s-r)^{-\frac{1+\kappa'}{2}}(1+\|w(r)\|_{L^{\frac{2p+2}{3}}})dr\right\}^{\frac{2p+2}{3}}ds,\\
&\lesssim\left(\int_0^ts^{-\kappa'(p+1)}ds\right)^{\frac{2}{3}}
\left[\int_0^t\left\{\int_0^s(s-r)^{-\frac{1+\kappa'}{2}}a_r^{\frac{1}{2p+2}}dr\right\}^{2p+2}ds\right]^{\frac{1}{3}}\lesssim A_t^{\frac{1}{3}}.
\end{align*}
For the fifth term, we have
\begin{align*}
\int_0^t\left\{\int_0^s(s-r)^{-\frac{1+3\kappa'}{2}}\|w(r)\|_{\B_{\frac{2p+2}{3}}^{1+2\kappa'}}dr\right\}^{\frac{2p+2}{3}}ds\lesssim\int_0^t\|w(s)\|_{\B_{\frac{2p+2}{3}}^{1+2\kappa'}}^{\frac{2p+2}{3}}ds=C_t.
\end{align*}
For the last term, we need the following estimate.
\begin{align}\label{w:G_6:estimate of deltaw}
\|\delta_{st}w\|_{L^{\frac{2p+2}{3}}}\lesssim(t-s)^{2\kappa'}(1+\|v_0\|_{\B_{2p+2}^{\frac{1}{2}+\kappa'}}^{2p+2}+c_s+A_t+C_t)^{\frac{3}{2p+2}}.
\end{align}
Since the proof of this estimate requires many pages, we show it in the next section. Now we assume that \eqref{w:G_6:estimate of deltaw} is true. Let $N_t=1+\|v_0\|_{\B_{2p+2}^{\frac{1}{2}+\kappa'}}^{2p+2}+A_t+C_t$. Then for small $\delta>0$, we have
\begin{align*}
&\int_0^t\left\{\int_{s-\delta}^s(s-r)^{-1-\kappa'}\|\delta_{rs}w\|_{L^{\frac{2p+2}{3}}}dr\right\}^{\frac{2p+2}{3}}ds\\
&\lesssim\int_0^t\left\{\int_{s-\delta}^s(s-r)^{-1+\kappa'}(N_s+c_r)^{\frac{3}{2p+2}}dr\right\}^{\frac{2p+2}{3}}ds\\
&\lesssim\int_0^tN_s\left\{\int_{s-\delta}^s(s-r)^{-1+\kappa'}dr\right\}^{\frac{2p+2}{3}}ds+\int_0^tc_sds\lesssim\delta^{\frac{2p+2}{3}\kappa'}N_t+C_t.
\end{align*}
For the integral on $[0,s-\delta]$, we have
\begin{align*}
&\int_0^t\left\{\int_0^{s-\delta}(s-r)^{-1-\kappa'}\|\delta_{rs}w\|_{L^{\frac{2p+2}{3}}}dr\right\}^{\frac{2p+2}{3}}ds\\
&\lesssim\int_0^t\delta^{-\frac{2p+2}{3}(1+\kappa')}\left\{\int_0^s(\|w(s)\|_{L^{\frac{2p+2}{3}}}+\|w(r)\|_{L^{\frac{2p+2}{3}}})dr\right\}^{\frac{2p+2}{3}}ds
\lesssim\delta^{-\frac{2p+2}{3}(1+\kappa')}C_t.
\end{align*}
To sum up, we have
\[\int_0^t\|\text{com}(v,w)(s)\|_{\B_{\frac{2p+2}{3}}^{1+\kappa'}}^{\frac{2p+2}{3}}ds\lesssim\|v_0\|_{\B_{2p+2}^{\frac{1}{2}+\kappa'}}^{\frac{2p+2}{3}}+A_t^{\frac{1}{3}}+\delta^{\frac{2p+2}{3}\kappa'} N_t+C_\delta C_t.\]
These complete the proof.
\end{proof}

The following lemma is obtained similarly to \cite[Lemmas~5.6 and 5.7]{MW16}, so we omit the proof.

\begin{lemm}\label{w:estimateI7}
For every $p>1$ and $\epsilon>0$, we have
\begin{align*}
\mathcal{I}_{(2)}(t)&
\le C_\epsilon+\epsilon(\|v_0\|_{\B_{2p+2}^{\frac{1}{2}+\kappa'}}^{2p+2}+A_t+B_t),\\
\mathcal{I}_{(3)}(t)&
\le C_\epsilon+\epsilon(\|v_0\|_{\B_{2p+2}^{\frac{1}{2}+\kappa'}}^{2p+2}+A_t+B_t),\\
\mathcal{I}_{(4)}(t)&
\le C_\epsilon+\epsilon(A_t+B_t),\\
\mathcal{I}_{(7)}(t)&
\le C_\epsilon C_t+\epsilon A_t.
\end{align*}
\end{lemm}

Now we can obtain Theorem \ref{w:control L^p by Besov} by combining these bounds and choosing small $\epsilon$ compared with $\delta\mu$ and $\Re\nu$ in \eqref{w:L^p inequality}.

\section{A priori estimate of $\delta w$}\label{section:apriori_dw}

In this section, we show \eqref{w:G_6:estimate of deltaw} and complete the proof of Theorem \ref{w:control L^p by Besov}. We can obtain a simpler result than \cite[Theorem~4.1]{MW16}.

\begin{theo}\label{dw:goal}
Let $p>1$ be such that $\frac{5}{2}\kappa'\le\frac{3}{4}-\frac{3}{2p+2}$. For $0\le s\le t\le T$, we have
\begin{align*}
\|\delta_{st}w\|_{L^{\frac{2p+2}3}}\lesssim(t-s)^{2\kappa'}(1+\|v_0\|_{\B_{2p+2}^{\frac{1}{2}+\kappa'}}^{2p+2}+\|w(s)\|_{\B_{\frac{2p+2}{3}}^{1+2\kappa'}}^{\frac{2p+2}{3}}+A_t+C_t)^{\frac{3}{2p+2}},
\end{align*}
where the implicit constant depends only on $\mu,\nu,\lambda,c,\kappa,\kappa',p,T$ and $\$\mathbb{X}\$_{\kappa,T}$.
\end{theo}

As discussed in \cite[Section~4]{MW16}, since
\begin{align}\label{dw:estimate of etw-w}
\|(e^{(t-s)L_\mu}-1)w(s)\|_{L^{\frac{2p+2}{3}}}\lesssim(t-s)^{2\kappa'}c_s^{\frac{3}{2p+2}},
\end{align}
it is sufficient to consider the estimate of
$$
\delta_{st}'w:=\delta_{st}w-(e^{(t-s)L_\mu}-1)w(s)=w(t)-e^{(t-s)L_\mu}w(s).
$$
We can decompose it as
\begin{align}\label{dw:decomposition of delta'w}
\delta_{st}'w=\sum_{i=1}^8\int_s^te^{(t-r)(i+\mu)\Delta}G_{(i)}(v,w)(r)=:\sum_{i=1}^8\mathcal{W}_{(i)}(s,t).
\end{align}
For simplicity, we write
\[q=\frac{2p+2}{3}\]
in what follows.

\begin{lemm}
For every $q>\frac{4}{3}$, we have
\begin{align*}
\|\mathcal{W}_{(1)}(s,t)\|_{L^q}&\lesssim(t-s)^{\frac{q-1}{q}}(\|v_0\|_{L^{3q}}^{3q}+A_t)^{\frac{1}{q}},\\
\|\mathcal{W}_{(5)}(s,t)\|_{L^q}&\lesssim(t-s)^{\frac{q-1}{q}}(\|v_0\|_{\B_{3q}^{\frac{1}{2}+\kappa'}}^q+C_t)^{\frac{1}{q}},\\
\|\mathcal{W}_{(7)}(s,t)\|_{L^q}&\lesssim(t-s)^{\frac{q-1}{q}}C_t^{\frac{1}{q}},\\
\|\mathcal{W}_{(8)}(s,t)\|_{L^q}&\lesssim(t-s)^{\frac{3}{4}-\frac{1}{q}}(\|v_0\|_{\B_{3q}^{\frac{1}{2}+\kappa'}}^q+C_t)^{\frac{1}{q}}.
\end{align*}
\end{lemm}

\begin{proof}
These are obtained by similar arguments to \cite[Lemmas~4.2 and 4.6]{MW16}. Here we prove only the last two assertions. For $\mathcal{W}_{(7)}$, we have
\begin{align*}
\|\mathcal{W}_{(7)}(s,t)\|_{L^q}
&\lesssim\int_s^t\|e^{(t-r)L_\mu}G_{(7)}(r)\|_{L^q}dr\lesssim\int_s^t\|G_{(7)}(r)\|_{L^q}dr\\
&\lesssim\left(\int_s^tdr\right)^{\frac{q-1}{q}}\left(\int_0^t\|w(r)\|_{\B_q^{1+2\kappa'}}^qdr\right)^{\frac{1}{q}}
\lesssim(t-s)^{\frac{q-1}{q}}C_t^{\frac{1}{q}}.
\end{align*}
For $\mathcal{W}_{(8)}$,
\begin{align*}
\|\mathcal{W}_{(8)}(s,t)\|_{L^q}
&\lesssim\int_s^t\|e^{(t-r)L_\mu}G_{(8)}(r)\|_{\B_q^{\kappa'-\kappa}}dr\lesssim\int_s^t(t-r)^{-\frac{1}{4}}\|G_{(8)}(r)\|_{\B_q^{-\frac{1}{2}+\kappa'-\kappa}}dr\\
&\lesssim\left(\int_s^t(t-r)^{-\frac{1}{4}\frac{q}{q-1}}dr\right)^{\frac{q-1}{q}}\left\{\int_0^t(\|v(r)\|_{\B_q^{\frac{1}{2}+\kappa'}}+\|w(r)\|_{\B_q^{\frac{1}{2}+\kappa'}})^qdr\right\}^{\frac{1}{q}}.
\end{align*}
The first factor is bounded by $(t-s)^{\frac{3}{4}-\frac{1}{q}}$ because $q>\frac{4}{3}$. We can show that the time integral of $\|v\|_{\B_q^{\frac{1}{2}+\kappa'}}$ is bounded by
\[\|v_0\|_{\B_q^{\frac{1}{2}+\kappa'}}^q+C_t,\]
as already discussed above.
\end{proof}

\begin{lemm}\label{dw:estimate of W_2-4}
For every $q>\frac{4}{3}$ such that $\frac{1}{q}<\frac{3}{4}-\frac{1}{2}\kappa'$, we have
\begin{align*}
\|\mathcal{W}_{(2)}(s,t)\|_{L^q}&\lesssim(t-s)^{\frac{3}{4}-\frac{1}{q}-\frac{1}{2}\kappa'}(\|v_0\|_{\B_{3q}^{\frac{1}{2}+\kappa'}}^{3q}+A_t+C_t)^{\frac{1}{q}},\\
\|\mathcal{W}_{(3)}(s,t)\|_{L^q}&\lesssim(t-s)^{\frac{3}{4}-\frac{1}{q}-\frac{1}{2}\kappa'}(\|v_0\|_{\B_{3q}^{\frac{1}{2}+\kappa'}}^q+C_t)^{\frac{1}{q}},\\
\|\mathcal{W}_{(4)}(s,t)\|_{L^q}&\lesssim(t-s)^{\frac{3}{4}-\kappa'}.
\end{align*}
\end{lemm}

\begin{proof}
We now focus on the first one. The others are obtained by similar arguments. We start with the estimate
\begin{align*}
\|\mathcal{W}_{(2)}(s,t)\|_{L^q}
&\lesssim\int_s^t\|e^{(t-r)L_\mu}G_{(2)}(r)\|_{\B_q^{\kappa'-\kappa}}dr\lesssim\int_s^t(t-r)^{-\frac{1+2\kappa'}{4}}\|G_{(2)}(r)\|_{\B_q^{-\frac{1}{2}-\kappa}}dr\\
&\lesssim\left(\int_s^t(t-r)^{-\frac{1+2\kappa'}{4}\frac{q}{q-1}}dr\right)^{\frac{q-1}{q}}\left(\int_0^t\|G_{(2)}(r)\|_{\B_q^{-\frac{1}{2}-\kappa}}^qdr\right)^{\frac{1}{q}}.
\end{align*}
We will show the bound
\begin{align}\label{w:W_2:estimate of G_2}
\|G_{(2)}(r)\|_{\B_q^{-\frac{1}{2}-\kappa}}^q
&\lesssim\|(v+w)^2(r)\|_{\B_q^{\frac{1}{2}+\kappa'}}^q+\|(v+w)(\conj{v}+\conj{w})(r)\|_{\B_q^{\frac{1}{2}+\kappa'}}^q\\\notag
&\lesssim\|v(r)\|_{\B_{3q}^{\frac{1}{2}+\kappa'}}^{3q}+a_r+c_r
\end{align}
by estimating the terms involving (1) $v^2,v\conj{v}$, (2) $vw,v\conj{w},\conj{v}w$ and (3) $w^2,w\conj{w}$ separately. For (1), we have
$$
\|v^2\|_{\B_q^{\frac{1}{2}+\kappa'}}+\|v\conj{v}\|_{\B_q^{\frac{1}{2}+\kappa'}}\lesssim\|v\|_{\B_{2q}^{\frac{1}{2}+\kappa'}}^2\lesssim1+\|v\|_{\B_{3q}^{\frac{1}{2}+\kappa'}}^3.
$$
For (2), by Young's inequality and the interpolation (Lemma \ref{para:interpolation}) we have
\begin{align*}
\|vw\|_{\B_q^{\frac{1}{2}+\kappa'}}+\|v\conj{w}\|_{\B_q^{\frac{1}{2}+\kappa'}}
&\lesssim\|v\|_{\B_{3q}^{\frac{1}{2}+\kappa'}}\|w\|_{\B_{\frac{3q}{2}}^{\frac{1}{2}+\kappa'}}
\lesssim\|v\|_{\B_{3q}^{\frac{1}{2}+\kappa'}}\|w\|_{L^{3q}}^{\frac{1}{2}}\|w\|_{\B_q^{1+2\kappa'}}^{\frac{1}{2}}\\
&\lesssim\|v\|_{\B_{3q}^{\frac{1}{2}+\kappa'}}^3+\|w\|_{L^{3q}}^3+\|w\|_{\B_q^{1+2\kappa'}}.
\end{align*}
For (3), by Bony's decomposition
$$
w^2=w\rs w+2w\pl w
$$
we have
\begin{align*}
\|w^2\|_{\B_q^{\frac{1}{2}+\kappa'}}
&\lesssim\|w\pl w\|_{\B_q^{\frac{1}{2}+\kappa'}}+\|w\rs w\|_{\B_q^{\frac{1}{2}+\kappa'}}
\lesssim\|w\|_{L^{3q}}\|w\|_{\B_{\frac{3q}{2}}^{\frac{1}{2}+\kappa'}}+\|w\|_{\B_{2q}^{\frac{1}{4}+\frac{1}{2}\kappa'}}^2\\
&\lesssim\|w\|_{L^{3q}}(\|w\|_{L^{3q}}^{\frac{1}{2}}\|w\|_{\B_q^{1+2\kappa'}}^{\frac{1}{2}})+(\|w\|_{L^{3q}}^{\frac{3}{4}}\|w\|_{\B_q^{1+2\kappa'}}^{\frac{1}{4}})^2\\
&\lesssim\|w\|_{L^{3q}}^3+\|w\|_{\B_q^{1+2\kappa'}}.
\end{align*}
Now we get the required bounds because
\begin{align*}
\int_0^t\|v(r)\|_{\B_{3q}^{\frac{1}{2}+\kappa'}}^{3q}dr
\lesssim\|v_0\|_{\B_{3q}^{\frac{1}{2}+\kappa'}}^{3q}+A_t
\end{align*}
by \eqref{v:estimate of v in Besov}. These complete the proof.
\end{proof}

\begin{lemm}
For every $q>\frac{4}{3}$ such that $\frac{1}{q}<1-\kappa'$, we have
\begin{align*}
\|\mathcal{W}_{(6)}(s,t)\|_{L^q}\lesssim(t-s)^{1-\frac{1}{q}-\kappa'}(\|v_0\|_{\B_q^{\frac{1}{2}+\kappa'}}^q+C_t+\$w\$_{q,\kappa';t}^{\frac{q}{2}}C_t^{\frac{1}{2}})^{\frac{1}{q}},
\end{align*}
where
\[\$w\$_{q,\kappa';t}:=\sup_{0\le u<r\le t}\frac{\|\delta_{ur}'w\|_{L^q}}{|r-u|^{2\kappa'}}.\]
\end{lemm}

\begin{proof}
Since
\begin{align*}
\|\mathcal{W}_{(6)}(s,t)\|_{L^q}\lesssim\int_s^t\|G_{(6)}(r)\|_{\B_q^{\kappa'-\kappa}}dr\lesssim\int_s^t\|\com(v,w)(r)\|_{\B_q^{1+\kappa'}}dr,
\end{align*}
we consider the time integral of each term in \eqref{v:estimate of com in Besov}. For the first two terms, we have
\[\int_s^t(1+r^{-\frac{1}{4}}\|v_0\|_{\B_q^{\frac{1}{2}+\kappa'}})dr\lesssim(t-s)^{\frac{3}{4}}(1+\|v_0\|_{\B_q^{\frac{1}{2}+\kappa'}}).\]
For the next three terms, since $\kappa'\frac{q}{q-1}<4\kappa'<1$ we have
\begin{align*}
&\int_s^tr^{-\kappa'}(1+\|w(r)\|_{L^q})dr+\int_s^tr^{-\kappa'}\int_0^r(r-u)^{-\frac{1+\kappa'}{2}}(1+\|w(u)\|_{L^q})du\,dr\\
&\le\left(\int_s^tr^{-\kappa'\frac{q}{q-1}}dr\right)^{\frac{q-1}{q}}\left[\left(\int_0^tc_rdr\right)^{\frac{1}{q}}+\left(\int_0^t\int_0^r(r-u)^{-\frac{1+\kappa'}{2}}c_udu\,dr\right)^{\frac{1}{q}}\right]\\
&\lesssim(t-s)^{1-\frac{1}{q}-\kappa'}C_t^{\frac{1}{q}},
\end{align*}
and
\begin{align*}
&\int_s^t\int_0^r(r-u)^{-\frac{1+3\kappa'}{2}}\|w(u)\|_{\B_q^{1+2\kappa'}}du\,dr\\
&\lesssim\left(\int_s^tdr\right)^{\frac{q-1}{q}}\left(\int_0^t\int_0^r(r-u)^{-\frac{1+3\kappa'}{2}}c_udu\,dr\right)^{\frac{1}{q}}\\
&\lesssim(t-s)^{\frac{q-1}{q}}C_t^{\frac{1}{q}}.
\end{align*}
For the last term, we can replace $\delta_{st}w$ by $\delta_{st}'w$ since the difference is estimated by
\begin{align*}
&\int_s^t\int_0^r(r-u)^{-1-\frac{\kappa+\kappa'}{2}}|\|\delta_{ur}w\|_{L^q}-\|\delta_{ur}'w\|_{L^q}|du\,dr\\
&\lesssim\int_s^t\int_0^r(r-u)^{-1-\kappa'}(r-u)^{2\kappa'}c_u^{\frac{1}{q}}du\,dr\\
&\lesssim\left(\int_s^tdr\right)^{\frac{q-1}{q}}\left(\int_s^t\int_0^r(r-u)^{-1+\kappa'}c_udu\,dr\right)^{\frac{1}{q}}\lesssim(t-s)^{\frac{q-1}{q}}C_t^{\frac{1}{q}}.
\end{align*}
For the contribution of $\delta_{st}'w$, since
\begin{align*}
\|\delta_{ur}'w\|_{L^q}&\lesssim\|\delta_{ur}'w\|_{L^q}^{\frac{1}{2}}(\|w(r)\|_{L^q}^{\frac{1}{2}}+\|w(u)\|_{L^q}^{\frac{1}{2}})\\
&\le\$w\$_{q,\kappa';t}^{\frac{1}{2}}(r-u)^{\kappa'}(\|w(r)\|_{L^q}^{\frac{1}{2}}+\|w(u)\|_{L^q}^{\frac{1}{2}}),
\end{align*}
we have
\begin{align*}
&\int_s^t\int_0^r(r-u)^{-1-\frac{\kappa+\kappa'}{2}}\|\delta_{ur}'w\|_{L^q}du\,dr\\
&\lesssim\$w\$_{q,\kappa';t}^{\frac{1}{2}}\int_s^t\int_0^r(r-u)^{-1+\frac{\kappa'-\kappa}{2}}(\|w(r)\|_{L^q}^{\frac{1}{2}}+\|w(u)\|_{L^q}^{\frac{1}{2}})du\,dr\\
&\lesssim\$w\$_{q,\kappa';t}^{\frac{1}{2}}\left(\int_s^tdr\right)^{\frac{2q-1}{2q}}\left(\int_0^t\int_0^r(r-u)^{-1+\frac{\kappa'-\kappa}{2}}(c_r+c_u)du\,dr\right)^{\frac{1}{2q}}\\
&\lesssim\$w\$_{q,\kappa';t}^{\frac{1}{2}}(t-s)^{\frac{2q-1}{2q}}C_t^{\frac{1}{2q}}.
\end{align*}
\end{proof}

Combining these estimates, we obtain the required result.

\begin{proof}[{Proof of Theorem \ref{dw:goal}}]
By assumption of $\kappa'$, all of the exponents of $(t-s)$ appeared in the above estimates are greater than $2\kappa'$. To sum them up, we have
\begin{align*}
\|\delta_{st}'w\|_{L^q}\lesssim(t-s)^{2\kappa'}(1+\|v_0\|_{\B_{3q}^{\frac{1}{2}+\kappa'}}^{3q}+A_t+C_t+\$w\$_{q,\kappa';t}^{\frac{q}{2}}C_t^{\frac{1}{2}})^{\frac{1}{q}},
\end{align*}
which yields
\begin{align*}
\$w\$_{q,\kappa';t}^q\lesssim
1+\|v_0\|_{\B_{3q}^{\frac{1}{2}+\kappa'}}^{3q}+A_t+C_t+\$w\$_{q,\kappa';t}^{\frac{q}{2}}C_t^{\frac{1}{2}}.
\end{align*}
From the fact that $x\le a+\sqrt{bx}\Rightarrow x\lesssim a+b$, we have
\[\$w\$_{q,\kappa';t}\lesssim
(1+\|v_0\|_{\B_{3q}^{\frac{1}{2}+\kappa'}}^{3q}+A_t+C_t)^{\frac{1}{q}},\]
which implies Theorem \ref{dw:goal}.
\end{proof}

\section{A priori $L^1[0,T]$ estimate of $(v,w)$}\label{section:integrable}

The goal of this section is the following theorem. From now on, we always assume
$$
1<p<5\wedge\{1+\mu(\mu+\sqrt{1+\mu^2})\}.
$$

\begin{theo}\label{int:apriori estimate on L^1}
Assume that $\kappa'<\frac{2}{5}(\frac{3}{4}-\frac{1}{q})\wedge\frac{1}{18}$. Let $(v,w)$ be the solution of the system \eqref{solution:CGLsystem} with initial value $(v_0,w_0)\in\B_{3q}^{\frac12+\kappa'}\times\B_q^{\frac{3}{2}-2\kappa'}$. Then there exists a constant $C<\infty$ depending only on $\mu,\nu,\lambda,c,\kappa,\kappa',p,T,\$\mathbb{X}\$_{\kappa,T}$ and $\|v_0\|_{\B_{3q}^{\frac12+\kappa'}}+\|w_0\|_{\B_q^{\frac{3}{2}-\kappa'}}$ such that
\[\int_0^T(\|w(t)\|_{\B_{q}^{1+2\kappa'}}^{q}+\|v(t)\|_{\B_{3q}^{\frac{1}{2}+\kappa'}}^{3q}+\|w(t)\|_{L^{3q}}^{3q})dt\le C.\]
\end{theo}

First we will show the follwing result.

\begin{lemm}\label{int:control c by I in short time}
There exist $T_*>0$ and $M<\infty$ depending only on $\mu,\nu,\lambda,c,\kappa,\kappa',p,T$ and $\$\mathbb{X}\$_{\kappa,T}$ such that for every $0\le s\le t\le T$ satisfying $t-s\le2T_*$,
\[\int_s^t\|w(r)\|_{\B_q^{1+2\kappa'}}^qdr\le M(1+\|w(s)\|_{\B_q^{1+2\kappa'}}^q+\|v(s)\|_{\B_{3q}^{\frac{1}{2}+\kappa'}}^{3q}+\|w(s)\|_{L^{3q}}^{3q}).\]
\end{lemm}

To prove the above lemma, we use the decomposition \eqref{dw:decomposition of delta'w} and write
\begin{align*}
\int_s^t\|w(r)\|_{\B_{q}^{1+2\kappa'}}^{q}dr
&\lesssim(t-s)\|w(s)\|_{\B_{q}^{1+2\kappa'}}^{q}
+\sum_{i=1}^8\int_s^t\|\mathcal{W}_{(i)}(s,r)\|_{\B_{q}^{1+2\kappa'}}^{q}dr.
\end{align*}
In Lemmas \ref{integrable:Lq[st] of 1-4}-\ref{integrable:Lq[st] of 578}, we will show that the last eight terms are bounded by the terms of the form:
\[(t-s)^\theta(\|v(s)\|_{\B_{3q}^{\frac{1}{2}+\kappa'}}^{3q}+V(s,t)+A(s,t)+C(s,t)),\quad\theta\in(0,1),\]
where
\begin{align*}
V(s,t)&=\int_s^t(1+\|v(r)\|_{\B_{3q}^{\frac{1}{2}+\kappa'}}^{3q})dr,\\
A(s,t)&=\int_s^t(1+\|w(r)\|_{L^{3q}}^{3q})dr,\\
C(s,t)&=\int_s^t\|w(r)\|_{\B_{q}^{1+2\kappa'}}^{q}dr.
\end{align*}
As discussed in \cite[Section~6]{MW16}, our proof starts with Young's convolution inequality. For $i=1,\dots,8$, we have
\begin{align*}
\int_s^t\|\mathcal{W}_{(i)}(s,r)\|_{\B_q^{1+2\kappa'}}^qdr
&\lesssim\int_s^t\left(\int_s^r(r-u)^{-\frac{1+2\kappa'-\alpha_i}{2}}\|G_{(i)}(u)\|_{\B_q^{\alpha_i}}du\right)^qdr\\
&\lesssim\left(\int_s^t(t-r)^{-\frac{1+2\kappa'-\alpha_i}{2}}dr\right)^q\int_s^t\|G_{(i)}(r)\|_{\B_q^{\alpha_i}}^qdr\\
&\lesssim(t-s)^{\frac{1-2\kappa'+\alpha_i}{2}q}\int_s^t\|G_{(i)}(r)\|_{\B_q^{\alpha_i}}^qdr,
\end{align*}
where $\alpha_i\in(-1+2\kappa',1+2\kappa')$. Thus we need to consider $L^q[s,t]$ estimates of $G_i$ in $\B_q^{\alpha_i}$ norm.

\begin{lemm}\label{integrable:Lq[st] of 1-4}
For every $0\le s\le t\le T$, we have
\begin{align*}
\int_s^t\|\mathcal{W}_{(1)}(s,r)\|_{\B_q^{1+2\kappa'}}^qdr&\lesssim(t-s)^{\frac{1-2\kappa'}{2}q}(V(s,t)+A(s,t)),\\
\int_s^t\|\mathcal{W}_{(2)}(s,r)\|_{\B_q^{1+2\kappa'}}^qdr&\lesssim(t-s)^{\frac{1-6\kappa'}{4}q}(V(s,t)+A(s,t)+C(s,t)),\\
\int_s^t\|\mathcal{W}_{(3)}(s,r)\|_{\B_q^{1+2\kappa'}}^qdr&\lesssim(t-s)^{\frac{1-6\kappa'}{4}q}(V(s,t)+C(s,t)),\\
\int_s^t\|\mathcal{W}_{(4)}(s,r)\|_{\B_q^{1+2\kappa'}}^qdr&\lesssim(t-s)^{\frac{1-8\kappa'}{4}q}\int_s^tdr.
\end{align*}
\end{lemm}

\begin{proof}
Let $\alpha_1=0$, $\alpha_2=\alpha_3=-\frac{1}{2}-\kappa$ and $\alpha_4=-\frac{1}{2}-2\kappa$. The first one immediately follows from
\[\|G_{(1)}(r)\|_{L^q}^q\lesssim\|v(r)\|_{L^{3q}}^{3q}+\|w(r)\|_{L^{3q}}^{3q}.\]
The second one follows from the bound \eqref{w:W_2:estimate of G_2}. The others are obtained more easily.
\end{proof}

\begin{lemm}
For every $0\le s\le t\le T$, we have
\begin{align*}
\int_s^t\|\mathcal{W}_{(6)}(s,r)\|_{\B_q^{1+2\kappa'}}^qdr&\lesssim(t-s)^{\frac{1-2\kappa'}{2}q}(\|v(s)\|_{\B_{3q}^{\frac{1}{2}+\kappa'}}^{3q}+A(s,t)+C(s,t)).
\end{align*}
\end{lemm}

\begin{proof}
Let $\alpha_6=0$. By the same argument as in the proof of Lemma \ref{w:lemm:estimate of G_6}, we have
\[\int_s^t\|G_{(6)}(r)\|_{L^q}^qdr\lesssim\int_s^t\|\com(v,w)(r)\|_{\B_q^{1+\kappa'}}^qdr\lesssim\|v(s)\|_{\B_{3q}^{\frac{1}{2}+\kappa'}}^{3q}+A(s,t)+C(s,t),\]
taking care that the initial time is $s$.
\end{proof}

\begin{lemm}\label{integrable:Lq[st] of 578}
For every $0\le s\le t\le T$, we have
\begin{align*}
\int_s^t\|\mathcal{W}_{(5)}(s,r)\|_{\B_q^{1+2\kappa'}}^qdr&\lesssim(t-s)^{\frac{1-2\kappa'}{2}q}(V(s,t)+C(s,t)),\\
\int_s^t\|\mathcal{W}_{(7)}(s,r)\|_{\B_q^{1+2\kappa'}}^qdr&\lesssim(t-s)^{\frac{1-2\kappa'}{2}q}C(s,t),\\
\int_s^t\|\mathcal{W}_{(8)}(s,r)\|_{\B_q^{1+2\kappa'}}^qdr&\lesssim(t-s)^{\frac{1-4\kappa'}{4}q}(V(s,t)+C(s,t)).
\end{align*}
\end{lemm}

\begin{proof}
Let $\alpha_5=\alpha_7=0$ and $\alpha_8=-\frac{1}{2}+\kappa'-\kappa$. The $L^q$ estimates of $G_{(5)},G_{(7)}$ and $G_{(8)}$ are easily obtained.
\end{proof}

To sum them up, we can show Lemma \ref{int:control c by I in short time}.

\begin{proof}[{Proof of Lemma \ref{int:control c by I in short time}}]
Combining above estimates, we have
\begin{align*}
C(s,t)\lesssim(t-s)\|w(s)\|_{\B_q^{1+2\kappa'}}^q+(t-s)^{\frac{1-8\kappa'}{4}}(\|v(s)\|_{\B_{3q}^{\frac{1}{2}+\kappa'}}^{3q}+V(s,t)+A(s,t)+C(s,t)).
\end{align*}
For $V$, from \eqref{v:estimate of v in Besov} we have
\begin{align}\label{int:estimate of V by A}
V(s,t)\lesssim1+\|v(s)\|_{\B_{3q}^{\frac{1}{2}+\kappa'}}^{3q}+A(s,t).
\end{align}
For $A$, we already have
\begin{align}\label{int:estimate of A by C}
A(s,t)\lesssim1+\|v(s)\|_{\B_{3q}^{\frac{1}{2}+\kappa'}}^{3q}+\|w(s)\|_{L^{3q}}^{3q}+C(s,t)
\end{align}
from Theorem \ref{w:control L^p by Besov}. Thus we have
\[C(s,t)\le M(t-s)\|w(s)\|_{\B_q^{1+2\kappa'}}^q+M(t-s)^{\frac{1-8\kappa'}{4}}(1+\|v(s)\|_{\B_{3q}^{\frac{1}{2}+\kappa'}}^{3q}+\|w(s)\|_{L^{3q}}^{3q}+C(s,t))\]
for some constant $M>0$. Therefore we obtain Lemma \ref{int:control c by I in short time} by choosing $T_*$ such that $M(2T_*)^{\frac{1-8\kappa'}{4}}\le\frac{1}{2}$.
\end{proof}

We return to the proof of Theorem \ref{int:apriori estimate on L^1}.

\begin{proof}[{Proof of Theorem \ref{int:apriori estimate on L^1}}]
Let
\[I_s=\|w(s)\|_{\B_q^{1+2\kappa'}}^q+\|v(s)\|_{\B_{3q}^{\frac{1}{2}+\kappa'}}^{3q}+\|w(s)\|_{L^{3q}}^{3q},\quad
I(s,t)=\int_s^tI_rdr.\]
By Combining Lemma \ref{int:control c by I in short time} with the estimates \eqref{int:estimate of V by A} and \eqref{int:estimate of A by C}, we have that for every $0\le s\le t\le T$ satisfying $t-s\le2T_*$,
\[I(s,t)\le M_*(1+I_s),\]
where $M_*$ depends only on $\mu,\nu,\lambda,c,\kappa,\kappa',p,T$ and $\$\mathbb{X}\$_{\kappa,T}$. Local well-posedness result (Theorem \ref{local:with singular initial condition} and Remark \ref{local:with smooth initial condition}) shows that there exist suitable choices of smaller $T_*$ and larger $M_*$, which depend on the initial value $(v_0,w_0)$, and such that we have
\[I(0,T_*)\le M_*.\]
For every $k\in\mathbf{N}$, because $I((k+1)T_*,(k+2)T_*)\le I(s,(k+2)T_*)$ for $s\in[kT_*,(k+1)T_*]$, we have
\begin{align*}
I((k+1)T_*,(k+2)T_*)&\le \frac{1}{T_*}\int_{kT_*}^{(k+1)T_*}I(s,(k+2)T_*)ds
\le\frac{M_*}{T_*}\int_{kT_*}^{(k+1)T_*}(1+I_s)ds\\
&\le M_*+\frac{M_*}{T_*}I(kT_*,(k+1)T_*).
\end{align*}
As a result, for $k=0,1,\dots$ we can prove that
\[I(kT_*,(k+1)T_*)\le M_*\sum_{i=0}^{k}\left(\frac{M_*}{T_*}\right)^i<\infty.\]
This completes the proof.
\end{proof}

\section{A priori $L^\infty[0,T]$ estimate of $(v,w)$}\label{section:global}

Let $(v,w)$ be the solution with initial value $(v_0,w_0)\in\B_{3q}^{\frac12+\kappa'}\times\B_q^{\frac{3}{2}-2\kappa'}$. In the settings of Theorem \ref{int:apriori estimate on L^1}, we show the following a priori $L^\infty[0,T]$ estimates of $(v,w)$.

\begin{theo}\label{global:L^infty of v}
Assume that $3\kappa'<\frac{3}{4}-\frac{1}{q}$. There exists a constant $C<\infty$ depending only on $\mu,\nu,\lambda,c,\kappa,\kappa',p,T,\$\mathbb{X}\$_{\kappa,T},\|v_0\|_{\B_{3q}^{\frac12+\kappa'}}$ and $\|w_0\|_{\B_q^{\frac{3}{2}-2\kappa'}}$ such that
\[\sup_{0\le t\le T}\|v(t)\|_{\B_{3q}^{\frac{1}{2}+\kappa'}}\le C.\]
\end{theo}

\begin{proof}
Since we already have a priori estimate $\int_0^T\|w(s)\|_{L^{3q}}^{3q}\lesssim1$ in Theorem \ref{int:apriori estimate on L^1}, from \eqref{v:estimate of v in Besov} we have
\begin{align*}
&\|v(t)\|_{\B_{3q}^{\frac{1}{2}+\kappa'}}\\
&\lesssim\|v_0\|_{\B_{3q}^{\frac{1}{2}+\kappa'}}+\left(\int_0^t(t-s)^{-(\frac{3}{4}+\kappa')\frac{3q}{3q-1}}ds\right)^{\frac{3q-1}{3q}}\left\{\int_0^t(1+\|w(s)\|_{L^{3q}})^{3q}ds\right\}^{\frac{1}{3q}}\\
&\lesssim1.
\end{align*}
since $(\frac{3}{4}+\kappa')\frac{3q}{3q-1}<1$.
\end{proof}

It remains to control $\|w\|_{\B_q^{\frac{3}{2}-2\kappa'}}$. We decompose it as follows.
\[\|w(t)\|_{\B_q^{\frac{3}{2}-2\kappa'}}\lesssim\|w_0\|_{\B_q^{\frac{3}{2}-2\kappa'}}+\sum_{i=1}^8\|\mathcal{W}_{(i)}(t)\|_{\B_q^{\frac{3}{2}-2\kappa'}}.\]
As discussed above, all of $\mathcal{W}_{(i)}$ have the bound of the form
\begin{align}\label{global:W by int of G}
\|\mathcal{W}_{(i)}(t)\|_{\B_q^{\frac{3}{2}-2\kappa'}}\lesssim\int_0^t(t-s)^{-\frac{3-4\kappa'-2\alpha_i}{4}}\|G_{(i)}(s)\|_{\B_q^{\alpha_i}}ds.
\end{align}
By Young's convolution inequality, we have
\[\left(\int_0^T\|\mathcal{W}_{(i)}(t)\|_{\B_q^{\frac{3}{2}-2\kappa'}}^{p_2}dt\right)^{\frac{1}{p_2}}\le\left(\int_0^T(T-t)^{-\frac{3-4\kappa'-2\alpha_i}{4}q_i}dt\right)^{\frac{1}{q_i}}\left(\int_0^T\|G_{(i)}(t)\|_{\B_q^{\alpha_i}}^{p_1}dt\right)^{\frac{1}{p_1}},\]
where $1+\frac{1}{p_2}=\frac{1}{q_i}+\frac{1}{p_1}$. This implies that if $\|G_{(i)}(t)\|_{\B_q^{\alpha_i}}$ has the $L^{p_1}[0,T]$ estimate, then we immediately have the $L^{p_2}$ estimate of $\|\mathcal{W}_{(i)}(t)\|_{\B_q^{\frac{3}{2}-2\kappa'}}$, where $q_i$ has to satisfy $\frac{3-4\kappa'-2\alpha_i}{4}q_i<1$. We ultimately aim to get $p_2=\infty$, which is interpreted as the $L^\infty[0,T]$ estimate: $\sup_{t\in[0,T]}\|w(t)\|_{\B_q^{\frac{3}{2}-2\kappa'}}<\infty$. Although this goal is not attained immediately, we are able to get $p_2=\infty$ by iterating Young's convolution inequality several times.

\begin{theo}\label{global:L^infty of w}
Assume that $q>\frac53$ and $\kappa'<\frac{1}{3}-\frac{1}{3q-2}$. There exists a constant $C<\infty$ depending only on $\mu,\nu,\lambda,c,\kappa,\kappa',p,T,\$\mathbb{X}\$_{\kappa,T},\|v_0\|_{\B_{3q}^{\frac12+\kappa'}}$ and $\|w_0\|_{\B_q^{\frac{3}{2}-2\kappa'}}$ such that
\[\sup_{0\le t\le T}\|w(t)\|_{\B_q^{\frac32-2\kappa'}}\le C.\]
\end{theo}

We start the proof by estimating each $\mathcal{W}_{(i)}$ using a priori estimates
\begin{align*}
&\sup_{t\in[0,T]}\|v(t)\|_{\B_{3q}^{\frac{1}{2}+\kappa'}}\lesssim1,\\
&\sup_{t\in[0,T]}\|w(t)\|_{L^{3q-2}}\lesssim1,\\
&\int_0^T\|w(t)\|_{\B_q^{1+2\kappa'}}^qdt\lesssim1,\\
&\int_0^T\|w(t)\|_{L^{3q}}^{3q}\lesssim1.
\end{align*}
We can improve the bounds of $\mathcal{W}_{(i)}$ as follows. Note that the proportional constants appearing above and in the following inequalities depend on initial values $(v_0,w_0)$.

\begin{lemm}
Assume that $3\kappa'<\frac{3}{4}-\frac{1}{q}$. For every $t\in[0,T]$,
\begin{align}
\label{global:estimate of W_1}\|\mathcal{W}_{(1)}(t)\|_{\B_q^{\frac{3}{2}-2\kappa'}}&\lesssim1+\int_0^t(t-s)^{-\frac{3-4\kappa'}{4}}\|w(s)\|_{L^{3q}}^3ds,\\
\|\mathcal{W}_{(2)}(t)\|_{\B_q^{\frac{3}{2}-2\kappa'}}&\lesssim1+\int_0^t(t-s)^{-\frac{2-\kappa'}{2}}(\|w(s)\|_{L^{3q}}^3+\|w(s)\|_{\B_q^{1+2\kappa'}})ds,\\
\|\mathcal{W}_{(3)}(t)\|_{\B_q^{\frac{3}{2}-2\kappa'}}&\lesssim1+\int_0^t(t-s)^{-\frac{2-\kappa'}{2}}\|w(s)\|_{\B_q^{\frac{1}{2}+\kappa'}}ds,\\
\|\mathcal{W}_{(4)}(t)\|_{\B_q^{\frac{3}{2}-2\kappa'}}&\lesssim1,\\
\|\mathcal{W}_{(5)}(t)\|_{\B_q^{\frac{3}{2}-2\kappa'}}&\lesssim1+\int_0^t(t-s)^{-\frac{3-4\kappa'}{4}}\|w(s)\|_{\B_q^{\frac{1}{2}+\kappa'}}ds,\\
\|\mathcal{W}_{(6)}(t)\|_{\B_q^{\frac{3}{2}-2\kappa'}}&\lesssim1+\int_0^t(t-s)^{-\frac{3-4\kappa'}{4}}\|w(s)\|_{\B_q^{1+2\kappa'}}ds,\\
\|\mathcal{W}_{(7)}(t)\|_{\B_q^{\frac{3}{2}-2\kappa'}}&\lesssim\int_0^t(t-s)^{-\frac{3-4\kappa'}{4}}\|w(s)\|_{\B_q^{1+\kappa'}}ds,\\
\label{global:estimate of W_8}\|\mathcal{W}_{(8)}(t)\|_{\B_q^{\frac{3}{2}-2\kappa'}}&\lesssim1+\int_0^t(t-s)^{-\frac{3-4\kappa'}{4}}\|w(s)\|_{\B_q^{1+\kappa'}}ds,
\end{align}
where the implicit constants depend on $\|v_0\|_{\B_{3q}^{\frac12+\kappa'}}+\|w_0\|_{\B_q^{\frac{3}{2}-2\kappa'}}$. As a result, we have
\begin{align}\label{global:w in L^q}
\int_0^T\|w(t)\|_{\B_q^{\frac{3}{2}-2\kappa'}}^qdt\lesssim1.
\end{align}
\end{lemm}

\begin{proof}
These are obtained by estimating $\|G_{(i)}\|_{\B_q^{\alpha_i}}$ in \eqref{global:W by int of G} as before. \eqref{global:w in L^q} is obtained by applying Jensen's inequality to \eqref{global:estimate of W_1}-\eqref{global:estimate of W_8}.
\end{proof}

We proceed to iterate Young's convolution inequality until we get $L^\infty[0,T]$ estimate. For simplicity, we write
\[\mathcal{W}_{(-1)}(t)=\sum_{i=2}^8\mathcal{W}_{(i)}(t).\]
Although $G_{(1)}$ and $G_{(2)}$ contain higher order terms of $w$, we can weaken their influence with the help of $L^\infty[0,T]$ estimate of $\|w(t)\|_{L^{3q-2}}$.

\begin{lemm}\label{global:lemm:iterating Young}
Assume that $q>\frac53$ and $\kappa'<\frac{1}{3}-\frac{1}{3q-2}$. For every $t\in[0,T]$,
\begin{align}
\label{global:W_1 by 12/7}\|\mathcal{W}_{(1)}(t)\|_{\B_q^{\frac{3}{2}-2\kappa'}}&\lesssim1+\int_0^t(t-s)^{-\frac{3-4\kappa'}{4}}\|w(s)\|_{\B_q^{\frac{3}{2}-2\kappa'}}^{\frac{12}{7}}ds,\\
\label{global:W_-1 by linear}\|\mathcal{W}_{(-1)}(t)\|_{\B_q^{\frac{3}{2}-2\kappa'}}&\lesssim1+\int_0^t(t-s)^{-\frac{2-\kappa'}{2}}\|w(s)\|_{\B_q^{\frac{3}{2}-2\kappa'}}ds.
\end{align}
If we assume that $\int_0^T\|w(t)\|_{\B_q^{\frac{3}{2}-2\kappa'}}^{p_1}ds\lesssim1$ for some $p_1\in[1,\infty)$, then we have
\begin{align}\label{global:improve W_1}
\int_0^T\|\mathcal{W}_{(1)}(t)\|_{\B_q^{\frac{3}{2}-2\kappa'}}^{p_2}ds\lesssim1
\end{align}
for $p_1>\frac{12}{7}$ such that $\frac{1}{p_2}>\frac{12}{7p_1}-\frac{1}{4}-\kappa'$, and we have
\begin{align}\label{global:improve W_-1}
\int_0^T\|\mathcal{W}_{(-1)}(t)\|_{\B_q^{\frac{3}{2}-2\kappa'}}^{p_3}ds\lesssim1
\end{align}
for $\frac{1}{p_3}>\frac{1}{p_1}-\frac{1}{2}\kappa'$.
\end{lemm}

\begin{proof}
For \eqref{global:W_1 by 12/7}, we need to replace $\|w\|_{L^{3q}}^{3q}$ by $\|w\|_{\B_q^{\frac{3}{2}-2\kappa'}}^{\frac{12}{7}}$. Indeed, for sufficiently small $\epsilon>0$,
\begin{align*}
\|w\|_{L^{3q}}\lesssim\|w\|_{\B_{3q}^\epsilon}\lesssim\|w\|_{\B_r^{\frac{7}{4}\epsilon}}^{\frac{4}{7}}\|w\|_{L^{3q-2}}^{\frac{3}{7}},
\end{align*}
where $r$ is determined by $\frac{1}{3q}=\frac{4}{7}\frac{1}{r}+\frac{3}{7}\frac{1}{3q-2}$. $\|w\|_{L^{3q-2}}$ is already bounded by $1$. Besov embedding shows $\B_q^{\frac{7}{4}\epsilon+3(\frac{1}{q}-\frac{1}{r})}\subset\B_r^{\frac{7}{4}\epsilon}$, where by assumption
\begin{align*}
\frac{7}{4}\epsilon+3\left(\frac{1}{q}-\frac{1}{r}\right)=\frac{7}{4}\epsilon+\frac{1}{4}\left(\frac{5}{q}+\frac{9}{3q-2}\right)<\frac{7}{4}\epsilon+\frac{5}{4q}+\frac{3}{4}-\frac{9}{4}\kappa'<\frac{3}{2}-2\kappa'
\end{align*}
for every $\epsilon<\frac17\kappa'$. Hence we have
\[\|w\|_{L^{3q}}\lesssim\|w\|_{\B_q^{\frac{3}{2}-2\kappa'}}^{\frac{4}{7}}.\]

For \eqref{global:W_-1 by linear}, it is sufficient to consider the square terms of $w$. As in the proof of Lemma \ref{dw:estimate of W_2-4}, by Bony's decomposition
\begin{align*}
\|w^2\|_{\B_q^{\frac{1}{2}+\kappa'}}&\lesssim\|w\pl w\|_{\B_q^{\frac{1}{2}+\kappa'}}+\|w\rs w\|_{\B_q^{\frac{1}{2}+\kappa'}}\lesssim\|w\|_{L^{3q-2}}\|w\|_{\B_r^{\frac{1}{2}+\kappa'}}+\|w\|_{\B_{2q}^{\frac{1}{4}+\frac{1}{2}\kappa'}}^2\\
&\lesssim\|w\|_{L^{3q-2}}\|w\|_{\B_r^{\frac{1}{2}+\kappa'}}+(\|w\|_{L^{3q-2}}^{\frac{1}{2}}\|w\|_{\B_r^{\frac{1}{2}+\kappa'}}^{\frac{1}{2}})^2\lesssim\|w\|_{L^{3q-2}}\|w\|_{\B_r^{\frac{1}{2}+\kappa'}},
\end{align*}
where $r$ is determined by $\frac{1}{q}=\frac{1}{3q-2}+\frac{1}{r}$. Boundedness of $\|w\|_{L^{3q-2}}$ and Besov embedding
\[\|w\|_{\B_r^{\frac{1}{2}+\kappa'}}\lesssim\|w\|_{\B_q^{\frac{3}{3q-2}+\frac{1}{2}+\kappa'}}\lesssim\|w\|_{\B_q^{\frac{3}{2}-2\kappa'}}\]
show the required estimate.

The improvement results \eqref{global:improve W_1}-\eqref{global:improve W_-1} are immediately obtained from Young's inequality. If $p_1>\frac{12}{7}$, we have
\begin{align*}
\left(\int_0^T\|\mathcal{W}_{(1)}(t)\|_{\B_q^{\frac{3}{2}-2\kappa'}}^{p_2}dt\right)^{\frac{1}{p_2}}\lesssim1+\left(\int_0^T(T-t)^{-\frac{3-4\kappa'}{4}r}dt\right)^{\frac{1}{r}}\left(\int_0^T\|w(t)\|_{\B_q^{\frac{3}{2}-2\kappa'}}^{p_1}dt\right)^{\frac{12}{7p_1}},
\end{align*}
where $1+\frac{1}{p_2}=\frac{1}{r}+\frac{12}{7p_1}$. Then (\ref{global:improve W_1}) follows if $\frac{3-4\kappa'}{4}r<1$, thus $\frac{1}{p_2}>\frac{12}{7p_1}-\frac{1}{4}-\kappa'$. \eqref{global:improve W_-1} is similar.
\end{proof}

By iterating this improvement result finite times (which depends on $\kappa'$), we obtain the required a priori estimate.

\begin{proof}[{Proof of Theorem \ref{global:L^infty of w}}]
First we show that we can replace the exponent $q$ in \eqref{global:w in L^q} by $q_1$, which satisfies $\frac{1}{q_1}=\frac{1}{q}-\frac{1}{4}$. From \eqref{global:estimate of W_1}, Young's inequality yields
\[\int_0^T\|\mathcal{W}_{(1)}(t)\|_{\B_q^{\frac{3}{2}-2\kappa'}}^{q_1}dt\lesssim1\]
because $1+\frac{1}{q_1}=\frac{3}{4}+\frac{1}{q}$. On the other hand, from Lemma \ref{global:lemm:iterating Young} we have $L^{p_1}[0,T]$ estimate of $\|\mathcal{W}_{(-1)}(t)\|_{\B_q^{\frac{3}{2}-2\kappa'}}$ with $\frac{1}{p_1}>\frac{1}{q}-\frac{1}{2}\kappa'$. To sum them up, we obtain $L^{p_1}[0,T]$ boundedness of $\|w(t)\|_{\B_q^{\frac{3}{2}-2\kappa'}}$. Now by applying Lemma \ref{global:lemm:iterating Young} again, we obtain $L^{p_2}[0,T]$ estimate of $\|\mathcal{W}_{(-1)}(t)\|_{\B_q^{\frac{3}{2}-2\kappa'}}$ with $\frac{1}{p_2}>\frac{1}{p_1}-\frac{1}{2}\kappa'>\frac{1}{q}-\kappa'$, which implies $L^{p_2}[0,T]$ boundedness of $\|w(t)\|_{\B_q^{\frac{3}{2}-2\kappa'}}$. We can repeat this argument until $p_{N_1}$ which satisfies $\frac{1}{p_{N_1}}<\frac{1}{q}-\frac{1}{4}$. ($N_1\sim\frac1{2\kappa'}$.) Hence we have
\[\int_0^T\|w(t)\|_{\B_q^{\frac{3}{2}-2\kappa'}}^{q_1}dt\lesssim1.\]

Next we show that we can again replace the exponent $q_1$ by $q_2$, which satisfies $\frac{1}{q_2}=\frac{12}{7q_1}-\frac{1}{4}$. We note that $\frac{1}{q_2}<\frac{1}{q_1}$ because $\frac{1}{q_1}<\frac{3}{5}-\frac{1}{4}=\frac{7}{20}$. Lemma \ref{global:lemm:iterating Young} implies
\[\int_0^T\|\mathcal{W}_{(1)}(t)\|_{\B_q^{\frac{3}{2}-2\kappa'}}^{q_2}dt\lesssim1.\]
Then by the same argument as above, we can conclude that $\mathcal{W}_{(-1)}$ is $L^{q_2}[0,T]$ bounded after performing $N_2$ ($\sim\frac2{\kappa'}(\frac1{q_1}-\frac1{q_2})$) times Young's inequalities, so $\|w(t)\|_{\B_q^{\frac{3}{2}-2\kappa'}}$ also.

We can replace the exponent $q_2$ by $q_3$ which satisfies $\frac{1}{q_3}=\frac{12}{7q_2}-\frac{1}{4}$ by the same arguments. We can repeat this argument until the sequence $\{\frac{1}{q_n}\}$ determined by
\[\frac{1}{q_{n+1}}=\frac{12}{7q_n}-\frac{1}{4}\]
achieves $\frac{1}{q_M}\le0$. ($M$ has the order $\mathcal{O}(|\log\kappa'|)$.) If $\frac{1}{q_M}<0$, then we should replace it by $q_M=\infty$. In the end, after performing $M+N_1+\cdots+N_M=\mathcal{O}((\kappa')^{-1})$ times improvements argument, we can complete the proof.
\end{proof}

\begin{proof}[{Proof of Theorem \ref{localexistence+:main theorem}}]
By Theorems \ref{global:L^infty of v} and \ref{global:L^infty of w}, we have a priori $L^\infty[0,T]$ estimate of $(v,w)$ if the conditions
$$
\frac32<p<5\wedge\{1+\mu(\mu+\sqrt{1+\mu^2})\},\quad\kappa'<\frac13-\frac1{2p}
$$
hold. Since
$$
\frac32<1+\mu(\mu+\sqrt{1+\mu^2})\Leftrightarrow\mu>\frac1{2\sqrt2},
$$
the assumption $\mu>\frac1{2\sqrt2}$ is satisfied if $p$ is sufficiently close to $\frac32$, or equivalently $\kappa'$ is sufficiently small.
\end{proof}

\section*{Acknowledgements}
This work was supported by JSPS KAKENHI, Grant-in-Aid for JSPS Fellows, 16J03010.  The author would like to thank Professor R. Fukuizumi for leading him to the problem discussed in this paper and Professor T. Funaki for him helpful remarks. He would like to thank Professor Y. Inahama and N. Naganuma also for their valuable advices.

\end{document}